\documentclass[12pt]{article}
\usepackage[utf8]{inputenc}

\usepackage[usenames]{xcolor}
\usepackage{amssymb}
\usepackage{latexsym}
\usepackage{amsmath}
\usepackage{amsthm}
\usepackage{amsopn}
\usepackage{lineno}
\usepackage{enumerate}

\theoremstyle{plain}
\newtheorem{theorem}{Theorem}
\newtheorem{lemma}{Lemma}

\newtheorem{corollary}[lemma]{Corollary}
\theoremstyle{remark}
\newtheorem{remark}{Remark}
\theoremstyle{definition}
\newtheorem{definition}{Definition}

\newcommand{\fofl}[1]{\mathrm{#1}}

\newcommand{\recfun}[1]{\mathtt{#1}}
\newcommand{\fast}{\recfun{F}}

\newcommand{\thr}[1]{\mathsf{#1}}
\newcommand{\PA}{\thr{PA}}
\newcommand{\EA}{\thr{EA}}
\newcommand{\IDO}{{\thr{I}\Delta_0}}
\newcommand{\IDOE}{{\IDO{+}\fofl{Exp}}}
\newcommand{\ISi}[1]{{\thr{I}\Sigma_{#1}}}

\newcommand{\PAsl}[1]{{\thr{PA}{\upharpoonright}_{#1}}}

\newcommand{\mthr}[1]{\mathsf{#1}}
\newcommand{\GL}{\mthr{GL}}
\newcommand{\GLT}{\mthr{GLT}}
\newcommand{\GLtwo}{\GL\mathsf{2}}

\newbox\gnBoxA
\newdimen\gnCornerHgt
\setbox\gnBoxA=\hbox{$\ulcorner$}
\global\gnCornerHgt=\ht\gnBoxA
\newdimen\gnArgHgt
\def\gnmb #1{%
\setbox\gnBoxA=\hbox{$#1$}%
\gnArgHgt=\ht\gnBoxA%
\ifnum     \gnArgHgt<\gnCornerHgt \gnArgHgt=0pt%
\else \advance \gnArgHgt by -\gnCornerHgt%
\fi \raise\gnArgHgt\hbox{$\ulcorner$} \box\gnBoxA %
\raise\gnArgHgt\hbox{$\urcorner$}}

\newcommand{\model}[1]{\mathfrak{#1}}
\newcommand{\krmod}[1]{\mathcal{#1}}

\newcommand{\TI}[2]{{\fofl{TI}_{#1}\mbox{-}#2}}

\newcommand{\conv}[1]{{#1{\downarrow}}}
\newcommand{\dive}[1]{{#1{\uparrow}}}

\newcommand{\stepdown}[2][]{\mathrel{\overset{#1}{\underset{#2}{\longrightarrow}}}}

\newcommand{\BoxI}[1][\,]{\raisebox{0pt}{\ooalign{\hfil$\vcenter{\large\hbox{$\square$}}$\hfil\cr
    \hfil$\vcenter{\scriptsize \hbox{$#1$}}$\hfil}}}
\newcommand{\DiamondI}[1][\,]{\raisebox{0pt}{\ooalign{\hfil$\vcenter{\large\hbox{\raisebox{-11pt}{$\Diamond$}}}$\hfil\cr
    \hfil$\vcenter{\scriptsize \hbox{$#1$}}$\hfil}}}

\renewcommand{\Box}{\square}

\newcommand{\BoxSl}{\boxdot}
\newcommand{\DiamondSl}{\ooalign{\hss$\Diamond$\hss\cr \kern0.48ex\raise0.1ex\hbox{$\cdot$}}}

\newcommand{\triangleDual}{\hbox{\raisebox{2.2pt}{$\bigtriangledown$}}}

\makeatletter
\newcommand{\dotminus}{\mathbin{\text{\@dotminus}}}

\newcommand{\@dotminus}{%
  \ooalign{\hidewidth\raise1ex\hbox{.}\hidewidth\cr$\m@th-$\cr}%
}
\makeatother

\begin{document}

\title{Slow and Ordinary Provability for \\ Peano Arithmetic}
\author{Paula~Henk\footnote{Institute of Logic, Language, and Computation, Amsterdam. E-mail: \texttt{p.henk@uva.nl}}\; and Fedor~Pakhomov\footnote{Steklov Mathematical Institute, Moscow. E-mail: \texttt{pakhfn@mi.ras.ru}. }{$\,\,$}\footnote{The work of the second author was supported by the Russian Foundation for Basic Research (grant number 15-01-09218).}}

\date{2016}
\maketitle

\begin{abstract} 
The notion of slow provability for Peano Arithmetic ($\PA$) was introduced by 
S.-D.\ Friedman, M.\ Rathjen, and A.\ Weiermann.  They studied the slow consistency statement $\fofl{Con}_{\mathsf{s}}$ asserting that a contradiction is not slow provable in $\PA$.
They showed that
the logical strength of the theory  $\PA+\fofl{Con}_{\mathsf{s}}$ lies strictly between that of $\PA$, and $\PA$ together with its ordinary consistency:
$\PA\subsetneq \PA+\fofl{Con}_{\mathsf{s}}\subsetneq \PA+\fofl{Con}_{\PA}$.

This paper is a further investigation into slow provability and its interplay with ordinary provability in $\PA$. We study three variants of slow provability. 
The associated consistency statement of each of these yields a theory that 
lies strictly between 
$\PA$ and $\PA+\fofl{Con}_{\PA}$ in terms of logical strength. We investigate Turing-Feferman progressions based on these variants of slow provability. For our three notions, the Turing-Feferman progression reaches 
$\PA+\fofl{Con}_{\PA}$ in a different numbers of steps, namely $\varepsilon_0$, $\omega$, and $2$.
For each of the three slow provability predicates, we also determine its joint provability logic with ordinary $\PA$-provability. 



 \end{abstract}

\section{Introduction}
Slow provability, introduced by S.-D.\ Friedman, $\mbox{M.\ Rathjen}$, and A.\ Weiermann in \cite{FriRatWei13}, is a  
notion of \emph{nonstandard provability} for Peano Arithmetic ($\PA$) -- while \emph{we} know that it 
coincides with ordinary provability for $\PA$, this fact is not verifiable in $\PA$ itself. This paper is a further investigation into the relation between slow and ordinary provability, as seen from the perspective of $\PA$.

The definition of slow provability relies on a \emph{fast-growing hiearachy}, also known as the \emph{extended Grzegorczyk hieararchy}. What we mean by this is, following \cite{FriRatWei13}, an 
ordinal-indexed
family of recursive functions $\{\fast_\alpha\}_{\alpha\leq \varepsilon_0}$.
The functions $\{\fast_\alpha\}_{\alpha <\omega}$ are closely related to a family of classes of functions known as the 
Grzegorczyk hierarchy (\cite{Grz55}). They are primitive recursive, and furthermore every primitive recursive function is dominated by some function in $\{\fast_\alpha\}_{\alpha <\omega}$. L\"ob and Wainer (\cite{LobWai70I}, \cite{LobWai70II})
extended the hierarchy into the transfinite. The exact version of the fast-growing hierarchy used in \cite{FriRatWei13} was introduced by Solovay and Ketonen (\cite{KetSol81}). 
The function $\fast_{\varepsilon_0}$ results from diagonalizing over the functions $\{\fast_{\omega_n}\}_{n\in\omega}$, 
each of which is provably total in $\PA$, and is not provably total in $\PA$ itself.  
This makes it interesting to consider the following r.e.\ theory: 
\begin{equation}\label{carcavelos}
\PAsl{\recfun{\fast_{\varepsilon_0}}} := \left\{\ISi{n}\mid \conv{\fast_{\varepsilon_0}(n)}\right\},
\end{equation}
where $\ISi{n}$ is as usual $\PA$ with the induction schema restricted to $\Sigma_n$-formulas. 
Since $\fast_{\varepsilon_0}$  is total, we know that 
$\PAsl{\fast_{\varepsilon_0}}$ and $\PA$ have exactly the same theorems. Arguing in $\PA$, on the other hand, 
the totality of $\fast_{\varepsilon_0}$ cannot be assumed, and thus
$\PAsl{\fast_{\varepsilon_0}}$ might seem to be a weaker theory than $\PA$.  As shown in \cite{FriRatWei13}, 
there exist indeed models of $\PA$ where a contradiction is provable in 
$\PA$ but not in $\PAsl{\recfun{\fast_{\varepsilon_0}}}$.


A notion of slow provability can be associated to any recursive function $\recfun{f}$ not provably total in $\PA$, 
by considering the 
theory 
\begin{equation}
\PAsl{\recfun{f}}:=\left\{\ISi{n}\mid \conv{\recfun{f}(n)}\right\}.
\end{equation} 
Since the equivalence of 
$\PA$ and $\PAsl{\recfun{f}}$ might not be verifiable in $\PA$, 
it is interesting 
to ask how exactly do the two theories relate to each other,
as seen from the perspective of $\PA$. This paper offers some ways of answering the above question. 

\subsection{Results of this paper}


Denote by $\fofl{Con}_{\PA}$, the usual consistency statement for $\PA$, and by 
$\fofl{Con}_{\recfun{f}}$ the statement expressing that a contradiction is not provable in $\PAsl{\recfun{f}}$.
As mentioned above, 
$\fofl{Con}_{\recfun{f}}$ need not be provably equivalent to $\fofl{Con}_{\PA}$.
However it is conceivable that by iterating $\fofl{Con}_{\recfun{f}}$ sufficiently many times, a statement 
equivalent to $\fofl{Con}_{\PA}$ is reached. We explore this possibility by
considering
\emph{transfinite iterations} of slow consistency statements. 
Given a non-zero
ordinal $\alpha\leq \varepsilon_0$, the $\alpha$-iteration $\fofl{Con}^{\alpha}_{\recfun{f}}$ of $\fofl{Con}_{\recfun{f}}$ is informally defined as the consistency statement for the theory
\begin{equation}\label{tr_con}
\PAsl{\recfun{f}} + \{\fofl{Con}^{\beta}_{\recfun{f}}\mid 0< \beta<\alpha\}.
\end{equation}
We adopt the provability logic approach to transfinite iterations, developing the 
notion of a transfinite iteration $\fofl{Pr}^{\alpha}(x)$ of a provability predicate $\fofl{Pr}(x)$ along a Kalmar elementary well-ordering. We show that these iterations satisfy the Hilbert-Bernays-Löb derivability conditions, and can thus be considered as provability predicates themselves (Section \ref{Transfinite_Section}). 

We show that the $\varepsilon_0$-iteration of $\fofl{Con_{\fast_{\varepsilon_0}}}$ is 
equivalent to $\fofl{Con}_{\PA}$ (Theorem \ref{Ite_Corollary}), thus answering a 
question raised in \cite[Remark 4.4]{FriRatWei13}. We also show that a
small index shift in the definition of
$\PAsl{\recfun{\fast_{\varepsilon_0}}}$ yields a slow consistency statement whose $\omega$-iteration
is already equivalent to  $\fofl{Con}_{\PA}$ (Theorem \ref{omegait}). While finishing writing this paper, the authors learned that the above results are also contained in Anton Freunds recent paper \cite{Fre16}. The results of our paper were obtained independently from the latter. 

We also introduce a variant of slow provability that can be seen as the \emph{square root} of ordinary 
$\PA$-provability, in the sense that 
already the two-fold iteration of the associated slow consistency statement 
is equivalent to $\fofl{Con}_{\PA}$ (Theorem \ref{Square_root_theorem}). Our slow provability variant is the first example of such a
provability predicate in the context of $\PA$.

For each of our three notions of slow provability, we determine its joint \emph{provability logic} with ordinary $\PA$-provability. 
While the slow provability predicate studied in \cite{FriRatWei13} and its shifted version mentioned above behave very differently when it comes to transfinite iterations, their joint provability logic with ordinary provability is the same, namely 
Lindstr\"om logic (Theorem \ref{GLTcompl}). It was shown in \cite{Lin06} that the
latter is also the joint provability logic of ordinary and Parikh provability, which can be seen as a \emph{speeded up} version of ordinary $\PA$-provability. Our proof or arithmetical completeness is rather general and works for a large class of pairs of provability predicates, including ordinary and Parikh provability.

\subsection{Overview of this paper}
Sections \ref{ArTheories_Section} and \ref{Ordinal_Section} contain the basic results and notions used in the paper.
Section \ref{Transfinite_Section} introduces
transfinite iterations of provability predicates. The notion of slow provability, along with some results from \cite{FriRatWei13}, 
forms the content of 
Section \ref{slowprovability}. 
In Section \ref{Converting_Section} we show that 
in some cases provability in  $\PA$ implies a certain transfinite iteration of slow provability. 
Section \ref{Models_Section} deals with the converse.  The joint provability logic of slow and ordinary provability is determined in Section \ref{ProvLogic_Section}. The material in this section relies only on sections \ref{ArTheories_Section} and \ref{slowprovability}.

\subsection{History and context}
We point out some developments related to the subject matter of this paper.  

\subsubsection{Nonstandard notions of provability for $\PA$}\label{nonstaPA}
The method of arithmetization developed by G\"odel allows $\PA$ to talk about basic syntactical notions. 
In particular, there is an arithmetical formula $\fofl{Pr}_{\thr{\PA}}(x)$, the so-called \emph{provability predicate}, that expresses basic facts about provability in $\thr{\PA}$. Writing $\fofl{Con}_{\thr{\PA}}$ for the sentence $\lnot \fofl{Pr}_{\thr{\PA}}(\bot)$, 
G\"odel's Second Incompleteness Theorem states that $\fofl{Con}_{\thr{\PA}}$ is not provable in $\PA$.

Since $\fofl{Pr}_{\thr{\PA}}(x)$ is, \emph{prima facie}, an arithmetical formula, one may justifiably ask what exactly is meant by calling it a provability predicate. Could there be another  provability predicate whose associated consistency statement \emph{is} provable in $\PA$? Likewise, which properties of $\fofl{Pr}_{\thr{\PA}}(x)$ does the proof of the Second Incompleteness Theorem rely on? 

Such questions were for the first time investigated by Feferman in his influential paper \cite{Fef60}. 
He constructs a predicate $\fofl{Pr}^{*}_{\PA}(x)$ that has the same extension as $\fofl{Pr}_{\thr{\PA}}(x)$ on the natural numbers, whose associated consistency statement is however provable in $\PA$.
The existence of such a \emph{nonstandard provability predicate} illustrates the need for a more careful formulation of the Second Incompleteness Theorem. 

In order to demonstrate the difficulty of singling out one ``standard" provability predicate for $\PA$, Feferman (\cite[Theorem 7.4, 7.5]{Fef60}) provides a rather
general method for modifying a given provability predicate $\fofl{Pr}(x)$, so as to obtain new provability predicates 
$\fofl{Pr}'(x)$ and $\fofl{Pr}''(x)$ for the same theory, whose associated consistency statements lie strictly below and above the original one respectively:
\begin{equation}
\PA\subsetneq \PA+\fofl{Con}'\subsetneq \PA+\fofl{Con} \subsetneq \PA+\fofl{Con}''.
\end{equation}
In particular, we obtain a theory lying between $\PA$ and $\PA+\fofl{Con}_{\PA}$ in terms of logical strength.
Since the above method relies on self-reference, in the form of a Rosser-style construction, it is reasonable to ask whether a \emph{natural} theory with this property can also be found. The theory $\PA+\fofl{Con_{\fast_{\varepsilon_0}}}$, obtained by adding to $\PA$ the statement of its slow consistency, may be seen as the first example of such a theory.

Another example of nonstandard provability is
the so-called \emph{Parikh provability}. An arithmetical sentence $\varphi$ is said to be Parikh provable if it is provable in $\PA$ together with Parkih's rule, where the latter allows 
one to infer $\varphi$ from the sentence $\fofl{Pr}_{\PA}(\varphi)$. 
Since Parikh's rule is admissible in $\PA$, adding it to $\PA$ does not yield new theorems. As shown in \cite{Par71}, it does yield speed-up, meaning that some theorems have much shorter proofs when Parikh's rule is allowed. The equivalence of 
Parikh provability and ordinary provability is however not verifiable in $\PA$.

\subsubsection{Slow provability for weak theories}
A notion of slow provability in the context of Kalmar Elementary Arithmetic ($\thr{EA}$) was introduced by Visser in 
\cite{Vis12}. He uses a superexponential (not provably total in $\thr{EA}$) function
in order to modify the standard provability predicate of $\thr{EA}$. As in the case of slow provability for $\PA$, the associated slow 
consistency statement 
lies strictly between $\thr{EA}$ and $\thr{EA}+\fofl{Con}_{\thr{EA}}$ in terms of logical strength.

A version of square root provability for
$\IDOE$ was also found by Visser: it is shown in \cite{Vis90} that cut-free or tableaux provability serve as the square root of ordinary provability in the context of $\IDOE$.

\subsubsection{Provability logic}
The idea of viewing $\fofl{Pr}_{\PA}(x)$ as a modal operator $\Box$ goes back to G\"odel (\cite{God33}). Hilbert and Bernays formulated certain conditions for $\fofl{Pr}_{\PA}(x)$ that would suffice for the proof of the Second Incompleteness Theorem. These were later simplified by L\"ob, and are now referred to as the \emph{Hilbert-Bernays-L\"ob derivability conditions}: 
\begin{enumerate}
\item $\PA \vdash \varphi \Rightarrow \PA \vdash \Box \varphi$
\item $\PA\vdash \Box(\varphi\to \psi) \to (\Box \varphi \to \Box \psi)$
\item $\PA\vdash \Box \varphi \to \Box \Box \varphi$
\end{enumerate}
The system $\GL$ of propositional modal logic is axiomatized by adding to the basic modal logic $\mathsf{K}$ the following, known as L\"ob's axiom: 
$\Box (\Box A\to A)\to \Box A$.
It was proven by Solovay (\cite{Sol76}) that
$\GL$ is the provability logic of $\PA$: its theorems are exactly the propositional schemata involving $\fofl{Pr}_{\PA}(x)$ that are provable in $\PA$. 

Solovay's method has been used to apply modal logic to study other metamathematical predicates besides $\fofl{Pr}_{\PA}(x)$. 
Shavrukov (\cite{Sha94}) determined  
the joint provability logic of ordinary and Feferman provability $\fofl{Pr}^*_{\PA}(x)$. The joint provability logic of ordinary and Parikh provability was established by Lindstr\"om (\cite{Lin06}).

\subsubsection{Turing-Feferman progressions}
The idea of transfinite iterations of consistency statements goes back to Turing (\cite{Tur39}).
Given a sufficiently strong $\Sigma_1$-sound theory $\thr{T}$, 
consider the sequence of theories given by: $\thr{T}^0:= \thr{T}$, and $\thr{T}^{n+1}:= \thr{T}^{n}+\fofl{Con}_{\thr{T}^{n}}$ for all $n$.
It follows from the Second Incompleteness Theorem that each $\thr{T}^{n+1}$ is a strictly stronger theory than $\thr{T}^{n}$.
In his doctoral thesis \cite{Tur39}, Turing introduced the method of transfinite iterations, allowing one to extend the above sequence into the transfinite.
Given a theory $\thr{T}$ and an ordinal $\alpha$, the theory $\thr{T}^{\alpha}$ is informally defined as: 
\begin{equation}\label{tr_ite}
\thr{T}^\alpha=\thr{T}+\fofl{Con}({\bigcup\limits_{\beta<\alpha}\thr{T}^{\beta}}).
\end{equation}
Returning to transfinite iterations of consistency statements introduced informally in (\ref{tr_con}) above, we note that
using the notation of (\ref{tr_ite}), we have that for 
$\alpha>0$,  
$\PAsl{\recfun{f}} + \fofl{Con}_{\recfun{f}}^{\alpha}$ is the theory $(\PAsl{\recfun{f}})^{\alpha}$. 

A proper construction of the above sequence of theories requires a \emph{recursive ordinal notation system}. 
As shown in 
\cite{Tur39,Fef62}, 
the properties of $\thr{T}^{\alpha}$ for infinite $\alpha$ depend significantly  on the choice of the ordinal notation system. These difficulties will not influence our paper, however, as we shall only consider ordinals 
$\alpha \le \varepsilon_0$, identifying any such $\alpha$ with its representation in Cantor normal form.


%

%
%


\section{Arithmetical theories}\label{ArTheories_Section}

We consider first-order theories formulated in the language $\mathcal{L}$ of arithmetic containing $0$, $\mathsf{S}$ (successor), $+$, $\times$, and $\leq$. As usual, an $\mathcal{L}$-formula is said to be
$\Delta_0$ (equivalently: $\Sigma_0$ or $\Pi_0$) if all its quantifiers are bounded, and 
$\Sigma_{n+1}$($\Pi_{n+1}$) if it is of the form $\exists x_0\ldots x_n\, \varphi$,
with $\varphi$ a $\Pi_n$($\Sigma_n$)-formula. We write $\overline{n}$ for the $\mathcal{L}$-term corresponding to $n$, i.e.\ $0$ followed by $n$ applications of $\mathsf{S}$. Given this, we shall mostly write $n$
instead of $\overline{n}$.

The basic facts concerning $0$, $\mathsf{S}$, $+$, $\times$, and $\leq$ are given by the axioms of the theory $\thr{Q}$ of Robinson Arithmetic 
 (\cite[Definition I.1.1] {HajPud98}). The theory $\thr{Q}$ is $\Sigma_1$-complete, meaning that it proves every true $\Sigma_1$-sentence. 

Our main interest in this article is the theory $\thr{\PA}$ of Peano Arithmetic that results from adding to $\thr{Q}$ the induction schema for all arithmetical formulas. 
As usual, $\ISi{n}$ denotes the fragment of $\PA$ obtained by restricting the induction schema to $\Sigma_{n}$-formulas.
Clearly, $\ISi{n}\subseteq\ISi{n+1}$ for all $n$, and 
$\PA=\bigcup_{n\in\omega} \ISi{n}$. We shall therefore sometimes also write $\ISi{\omega}$ for $\PA$.
Using that satisfaction for $\Sigma_{n}$-formulas is definable in $\ISi{1}$ by a $\Sigma_{n}$-formula, one can show that for all $n>0$, 
$\ISi{n}$ is finitely axiomatisable (\cite[Theorem I.2.52]{HajPud98}, see also Section \ref{meta_IDOE} below). 


As our metatheory, we mostly use $\IDOE$. In order to introduce the latter, we recall that there is a $\Delta_{0}$-formula 
$\varphi_{\recfun{e}}(x,y,z)$ that defines, provably in $\IDO$, the graph of a recursively defined exponentiation function $\recfun{e}(x,y)=x^y$ (\cite[Theorem V.3.15]{HajPud98}), i.e.\ we have:
\begin{align*}
\IDOE&\vdash \varphi_{\recfun{e}}(x,0,z) \leftrightarrow z=1\\
\IDOE&\vdash \varphi_{\recfun{e}}(x,y+1, z) \leftrightarrow \exists w\left(E(x,y,w) \land z=w \cdot x \right)
\end{align*}
The sentence stating the totality of this function:
\begin{equation}\label{exptot}
\forall x \forall y\exists ! z\, \varphi_{\recfun{e}}(x,y,z)
\end{equation}
is not provable in $\IDO$.
The theory $\IDOE$ is the result from adding (\ref{exptot}) as an additional axiom to $\IDO$.
We recall that $\IDOE$ is finitely axiomatizable (\cite[Theorem V.5.6 ]{HajPud98}). 

Since the formula defining exponentiation in $\IDO$ is $\Delta_{0}$, 
$\IDOE$ is a conservative extension of Elementary Arithmetic ($\EA$). By $\EA$, we mean the theory formulated in the language of arithmetic, together with a function symbol $\mathsf{exp}$ for exponentiation. It contains the basic facts concerning $0$, $\mathsf{S}$, $+$, $\times$, $\leq$ and $\mathsf{exp}$, plus induction for all $\Delta_{0}$-formulas of the extended language. $\EA$ is strong enough to formalize almost all of finitary mathematics outside logic.

\subsection{Representing recursive functions in $\IDOE$}\label{IDOErecfun}
It is well-known that a coding of sequences can be carried out inside $\IDOE$. Using that, it is straightforward to show that every primitive recursive relation 
$R$ can be represented inside $\IDOE$ by a $\Sigma_1$-formula $\varphi_R$, in the sense that for all $n_0, \ldots, n_k$, 
\begin{equation}
(n_0, \ldots, n_k)\in R \mbox{     iff     } \IDOE\vdash \varphi_R(\overline{n_0}, \ldots, \overline{n_k}).
\end{equation}
We recall that 
there are primitive recursive functions $T$ and $U$ with the property that for all recursive $\recfun{f}$, there exists some $e$, such that for all $n$,
\begin{equation}\label{univ}
\recfun{f}(n) = U\left(\mu y\, T(e, n, y)\right).
\end{equation}
Thus we can associate to any recursive function $\recfun{f}$ a $\Sigma_1$-formula $\varphi_\recfun{f}$
that defines $\recfun{f}$ in a natural way, say by mimicking its definition in (\ref{univ}).
If $\recfun{f}$ is $k$-ary, then
for all $n_1, \ldots, n_k$, we have that
\begin{align}
&\IDOE\vdash \varphi_{\recfun{f}}(\overline{n_1}, \ldots, \overline{n_k}, \overline{\recfun{f}(n_1, \ldots, n_k)})\\
&\IDOE\vdash \exists ! z\, \varphi_{\recfun{f}}(\overline{n_1}, \ldots, \overline{n_k}, z)
\end{align}
Since $\IDOE$ is $\Sigma_1$-sound, it follows that for any 
recursively enumerable (r.e.) set $A$, there is a $\Sigma_1$-formula $\varphi_A$ such that for all $n$, $n\in A$ if and only if $\IDOE\vdash \varphi_{A}(n)$. In fact, as was first shown in \cite{EhrFef60}, given any extension $\thr{S}$ of $\IDOE$ and any r.e.\ set $A$, there is a $\Sigma_1$-formula $\varphi_A$ such that for all $n$, $n\in A$ if and only if $\thr{S}\vdash \varphi_{A}(n)$.


Given a $k$-ary recursive function $\recfun{f}$, we denote by $\conv{\recfun{f} (x_1, \ldots, x_k)}$ the formula 
$\exists y\, \varphi_{\recfun{f}}(x_1, \ldots, x_k, y)$, 
and say that $\recfun{f}$ \emph{converges} on input $x_1, \ldots, x_k$. Similarly, we denote by 
$\dive{\recfun{f} (x_1, \ldots, x_k)}$ the formula $\lnot\conv{\recfun{f} (x_1, \ldots, x_k)}$, and say that $\recfun{f}$ \emph{diverges} on input $x_1, \ldots, x_n$. We use $\conv{\recfun{f}}$ as shorthand for
$$\forall x_1\ldots x_k\, \conv{\recfun{f} (x_1, \ldots, x_k)},$$ and $\dive{\recfun{f}}$ as shorthand for 
$\lnot\conv{\recfun{f}}$.  
  
A recursive function $\recfun{f}$ is said to be  \emph{provably recursive} in a theory $\thr{S}\supseteq \IDOE$ if 
$\thr{S}\vdash \conv{\recfun{f}}$.
The provably recursive functions of $\IDOE$ are exactly the Kalmar elementary functions. The class of Kalmar elementary functions is the smallest class containing successor, zero, projection, addition, multiplication, substraction, and closed under 
composition as well as bounded sums and bounded products (\cite{Ros84}).
For a characterization of the provably recursive functions of $\ISi{n}$ for $n \geq 1$, see 
Theorem \ref{prtot} in Section \ref{Ordinal_Section} below.

\subsection{Metamathematics in $\IDOE$}\label{meta_IDOE}

It is well-known that arithmetization of syntax can be carried out in $\IDOE$.
We assume as given some standard g\"odelnumbering of $\mathcal{L}$-formulas, and 
write $\gnmb{\varphi}$ for the g\"odelnumber of $\varphi$. 
We shall often identify a formula with its g\"odelnumber, writing $\psi(\varphi)$ instead of $\psi(\gnmb{\varphi})$.

Let $\thr{S}$ be a r.e.\ extension of $\IDOE$. As explained in Section \ref{IDOErecfun}, 
the set of axioms of $\thr{S}$ can be represented in $\IDOE$  by a $\Sigma_1$-formula $\varphi_{\thr{S}}$.
Using the latter, one can define in a natural way a $\Sigma_1$-formula $\mathrm{Pr}_{\varphi_{\thr{S}}}(x)$ representing provability in $\thr{S}$ inside $\IDOE$ (\cite[Definition 4.1]{Fef60}). In this paper, we shall write
$\mathrm{Pr}_{\thr{S}}$ instead of
 $\mathrm{Pr}_{\varphi_{\thr{S}}}$, having in mind some formula $\varphi_{\thr{S}}$ representing the axioms of $\thr{S}$ in $\IDOE$ in a natural way, by mimicking their informal definition.
We refer to 
$\mathrm{Pr}_{\thr{S}}$ as the \emph{standard provability predicate} of $\thr{S}$.  

We employ modal notation, writing $\Box_{\thr{S}}$ instead of $\mathrm{Pr}_{\thr{S}}$. We write $\Box_0$ as shorthand for 
$\Box_{\IDOE}$.
By $\Box_x$ we denote the formula containing $x$ as a free variable, and such that for $n>0$, $\Box_n$ (the result of substituting $\overline{n}$ for $x$ in $\Box_x$) is $\Box_{\ISi{n}}$. We write $\Box$, or sometimes also $\Box_{\omega}$, for $\Box_{\PA}$.
We write $\Diamond_{\thr{S}}\varphi$ as shorthand for 
$\lnot \Box_{\thr{S}} \lnot \varphi$.

We recall that $\PA$ is essentially reflexive, meaning that it proves the consistency of each of its finite subtheories, and the same holds for every consistent extension of $\PA$ in the language of arithmetic (\cite[Theorem III.2.35]{HajPud98}).

We use the dot notation as usual, thus $\Box_\thr{S}\varphi(\dot{x})$ means that the numeral for the value of $x$ has been substituted for the free variable of the formula $\varphi$ inside $\Box_{\thr{S}}$. If the intended meaning is clear from the context, we will often simply write $\Box_{\thr{S}} \varphi(x)$ instead of $\Box_\thr{S}\varphi(\dot{x})$.
We recall that any theory $\thr{S}$ extending $\IDOE$ is provably $\Sigma_1$-complete, meaning that for any $\Sigma_1$-formula $\sigma$, 
\begin{equation*}\label{idoesigma1}
\IDOE \vdash \sigma(x)\to \Box_\thr{S}\sigma(\dot{x}).
\end{equation*}
 It is well-known that if $\thr{S}$ is as above, then the 
\emph{Hilbert-Bernays-L\"ob derivability conditions} hold for $\Box_{\thr{S}}$ verifiably in $\IDOE$: \begin{enumerate}
\item if $\thr{S}\vdash \varphi$, then  $\IDOE\vdash \Box_{\thr{S}}\varphi$ \label{hbl1}
\item $\IDOE\vdash \Box_{\thr{S}}(\varphi\to \psi)\to (\Box_{\thr{S}} \varphi\to \Box_{\thr{S}}\psi)$ \label{hbl2}
\item $\IDOE\vdash \Box_{\thr{S}} \varphi \to \Box_{\thr{S}}\Box_{\thr{S}} \varphi$ \label{hbl3}
\end{enumerate}
We note that \ref{hbl2} and \ref{hbl3} also hold with internal variables ranging over $\varphi$ and $\psi$.

\begin{theorem}\label{diagonal_lemma}
Let $\varphi$ be an $\mathcal{L}$-formula whose free variables are exactly $x_0, \ldots, x_n$. Then there is an 
$\mathcal{L}$-formula $\psi$ with exactly the same free variables, and such that 
\begin{equation}
\IDOE \vdash \psi\left( x_1, \ldots, x_n\right) \leftrightarrow \varphi\left(\gnmb{\psi}, x_1, \ldots, x_n\right).
\end{equation}
\end{theorem}
From the proof of Theorem \ref{diagonal_lemma} it is clear that if $\varphi$ is $\Sigma_n$($\Pi_n$), then so is $\psi$.

Verifiability of L\"ob's principle for $\Box_{\thr{S}}$ in $\IDOE$ follows from the Hilbert-Bernays-L\"ob derivability conditions  for $\Box_{\thr{S}}$,
together with Theorem \ref{diagonal_lemma} (\cite[Theorem 3.2]{Boo93}).
This means that
\begin{equation}
\IDOE\vdash \Box_{\thr{S}}(\Box_{\thr{S}} \varphi\to \varphi)\to \Box_{\thr{S}} \varphi, 
\end{equation}
and thus 
modal principles valid in the G\"odel-L\"ob provability logic $\GL$ can be used when reasoning about $\thr{S}$ in $\IDOE$.

For $n\geq1$, there is a partial truth definition $\fofl{True}_{\Pi_n}(x)$ ($\fofl{True}_{\Sigma_n}(x)$) in $\IDOE$  for the class $\Pi_n$ ($\Sigma_n$) \cite[V.5(b)]{HajPud98}.  Thus for every $\varphi\in \Pi_n$ ($\varphi\in \Sigma_n$) we have $$\IDOE\vdash \varphi \mathrel{\leftrightarrow} \fofl{True}_{\Pi_n}(\varphi)\;\;\; (\IDOE\vdash \varphi \mathrel{\leftrightarrow} \fofl{True}_{\Sigma_n}(\varphi)).$$
Moreover,  $\fofl{True}_{\Pi_n}(x)$ and $\fofl{True}_{\Sigma_n}(x)$ satisfy Tarski's conditions (see \cite[Definition~I.1.74]{HajPud98}). For all $n\geq 1$, $\fofl{True}_{\Pi_n}(x)$ is a $\Pi_n$-formula, and 
 $\fofl{True}_{\Sigma_n}(x)$ is a $\Sigma_n$-formula. 

 Suppose $\alpha\in [1, \omega]$ and $n\ge 1$.  By $\BoxI[n]_{\alpha}$
  we denote the provability predicate for the theory $\ISi{\alpha}$ extended by all true $\Pi_n$ sentences.   We formalize $\BoxI[n]_{\alpha}\varphi$ in a natural way using a partial truth definition: $$\exists \psi \in \Pi_n\text{-Sen} ( \fofl{True}_{\Pi_n}(\psi) \land \Box_{\alpha}(\psi\to \varphi)),$$ where $\Pi_n\text{-Sen}$ denotes the set of all Gödel numbers of  $\Pi_n$-sentences. 
Here we can have quantifiers over $\alpha$ but not over $n$ in the language of arithmetic. We use $\DiamondI[n]_{\alpha}$ to denote the dual of $\BoxI[n]_{\alpha}$, i.e. 
$\DiamondI[n]_{\alpha}\varphi :=\lnot \BoxI[n]_{\alpha}\lnot\varphi$.  The sentence $\DiamondI[n]_{\alpha}\varphi$ is equivalent to uniform $\Pi_{n+1}$-reflection for $\ISi{\alpha}+\varphi$, i.e.\ the principle saying that if for some $\Pi_{n+1}$-formula $\psi(x)$ and every $m$ the theory $\ISi{\alpha}+\varphi$ proves $\psi(\overline{m})$, then $\forall x\, \psi(x)$ is true.


\section{Ordinals and the fast-growing hierarchy} \label{Ordinal_Section}
We introduce a certain fast-growing hierarchy of recursive functions indexed by ordinals below $\varepsilon_0$. We recall the basic facts concerning this hierarchy, including a characterization of the provably recursive functions of 
$\ISi{n}$, for \mbox{$n>0$}. 

In order to define the fast-growing functions, and to talk about them in our arithmetical theories, we need to represent ordinals below 
$\varepsilon_0$ as natural numbers. 
For that, it is useful to recall the Cantor normal form theorem: 
\begin{theorem}
For every ordinal $\alpha>0$, there exist unique $\alpha_0 \geq \alpha_1 \geq \ldots \geq \alpha_{k}$ with 
\begin{equation*}
\alpha= \omega^{\alpha_0} + \omega^{\alpha_1} + \ldots + \omega^{\alpha_k}.
\end{equation*}
\end{theorem}
The above representation of $\alpha$ is called its \emph{Cantor normal form}. 
Since $\varepsilon_0$ is the least ordinal $\varepsilon$ for which it holds that $\varepsilon=\omega^\varepsilon$, we see that if $\alpha<\varepsilon_0$, then $\alpha$ has a Cantor normal form with exponents $\alpha_i<\alpha$, and these exponents in turn have Cantor normal form with yet smaller exponents. We represent an ordinal $\alpha$ below $\varepsilon_0$ by either the symbol $0$ if $\alpha=0$, or otherwise its Cantor normal form $$ \omega^{\alpha_0} + \omega^{\alpha_1} + \ldots + \omega^{\alpha_k},$$ where each $\alpha_i$ is represented in the same way. More formally, this means that for any ordinal below $\varepsilon_0$, we fix a term built  $\omega^x$, $x+y$, and $0$. This method, known as Cantor ordinal notation system, is the most common way of representing ordinals below $\varepsilon_0$.

In order to work with the above terms in arithmetic, we represent them as their Gödel numbers. We note that the predicate $<$ and the standard functions of ordinal arithmetic ($x+y$, $x\cdot y$ and $\omega^x$) on Cantor ordinal notations can be expressed in  the language of arithmetic. Basic facts about ordinal arithmetic can be easily proven in $\IDOE$ (\cite[Section 3]{Som95}); we will omit the details of this formalization in our proofs. 

\subsection{The fast-growing hierarchy}

For an ordinal number $\alpha$ and $n<\omega$, we define $\omega^{\alpha}_{n}$ by 
$\omega^\alpha_0:=\alpha$, and $\omega^\alpha_{n+1}=\omega^{\omega^{\alpha}_n}$. 
We write $\omega_n$ 
for $\omega^1_n$.
Thus $\omega_0=1$, $\omega_{1}=\omega$, $\omega_{2}=\omega^{\omega}$, etc. It is well-known that the ordinal $\varepsilon_0$ can also be characterized as $\sup\{\omega_n\mid n\in\omega\}$; we therefore define $\omega_{\omega}:= \varepsilon_0$.

A \emph{fundamental sequence} for a countable limit ordinal $\lambda$ is a strictly monotone sequence $\{\lambda[n]\} _{n\in\omega}$ converging to $\lambda$ from below, i.e.\ $\lambda[n]<\lambda[n+1]<\lambda$ for all $n<\omega$, 
and $\sup\{\lambda[n]\mid n\in\omega\}=\lambda$. We consider the standard assignment of fundamental sequences to limit ordinals below $\varepsilon_0$.

\begin{definition}
Let $\varepsilon_0[n]:= \omega_{n+1}$.
For a limit ordinal $\lambda<\varepsilon_0$ with Cantor normal form 
$\lambda = \omega^{\alpha_0} + \omega^{\alpha_1} + \ldots + \omega^{\alpha_k}$,
we define $\lambda[n]$ as follows:
\begin{itemize}
\item If $\alpha_k$ is a successor ordinal, let $\lambda[n]:=\omega^{\alpha_0} + \omega^{\alpha_1} + \ldots + \omega^{(\alpha_{k}-1)}\cdot (n+1)$
\item if $\alpha_k$ is a limit ordinal, let $\lambda[n]:=\omega^{\alpha_0} + \omega^{\alpha_1} + \ldots + \omega^{\alpha_{k}[n]}$
\end{itemize}
\end{definition}
Given a function $\fast: \mathbb{N}\to \mathbb{N}$, we use exponential notation to denote repeated compositions of $\fast$, thus $\fast^0(x)=x$, and $\fast^{n+1}(x)=\fast(\fast^n(x))$. 
\begin{definition}\label{fasthier}
The fast-growing hierarchy $\{\fast_\alpha\}_{\alpha\leq \varepsilon_0}$ of recursive functions is given by: 
\begin{align*}
\fast_0(n)&=n+1\\
\fast_{\alpha+1}(n)&=\fast^{n+1}_\alpha(n) \\
\fast_{\lambda}(n)&= \fast_{\lambda[n]}(n)
\end{align*}
\end{definition}
This exact version of the fast-growing hierarchy was first introduced by Solovay and Ketonen in \cite{KetSol81}. Their results, together with results of Paris in \cite{Par80}, 
imply the following classification of the provably recursive functions of $\PA$: 

\begin{theorem}\label{prtot} For $n>0$, 
$\ISi{n}\vdash \conv{\fast_\alpha}  \;\; \Leftrightarrow\;\; \alpha<\omega_n$.
\end{theorem}

The computation of $\fast_\alpha(n)$ is closely connected to the following stepdown relation on ordinals.

\begin{definition} For any ordinals $\alpha,\beta\le \varepsilon_0$ and numbers $n,r$ we write $\alpha\stepdown[r]{n}\beta$ if there exists a sequence $\gamma_0,\ldots,\gamma_r$ such that $\gamma_0=\alpha$, $\gamma_r=\beta$, and for all $0\le i<r$, $\gamma_{i+1}=\gamma_{i}[n]$ if $\gamma_i$ is a limit ordinal and $\gamma_{i+1}+1=\gamma_i$, otherwise. We write $\alpha \stepdown{n}\beta$ in case 
$\alpha\stepdown[r]{n}\beta$ for some $r$.
\end{definition}

\begin{lemma}[$\IDOE$]\cite[Lemma~2.3, Lemma~2.4]{FriRatWei13}\label{fasthier_IDOE_facts}

\begin{enumerate}
\item If $\alpha\stepdown{n} \beta$ and $\conv{\fast_{\alpha}(n)}$ then $\conv{\fast_{\beta}(n)}$ and $\fast_{\alpha}(n)\ge\fast_{\beta}(n)$. 
\item If $\conv{\fast_{\alpha}(n)}$ and $n>m$, then $\conv{\fast_{\alpha}(m)}$ and $\fast_{\alpha}(n)\ge \fast_{\alpha}(m)$.
\item If $\alpha>\beta$ and $\conv{\fast_{\alpha}}$ then $\conv{\fast_{\beta}}$.
\item If $i>0$ and $\conv{\fast^i_{\alpha}(n)}$ then $n<\fast_{\alpha}^i(n)$. 
\item If $\conv{\fast_{\alpha}(n)}$ then $\alpha\stepdown{n} 0$.
\end{enumerate}
\end{lemma}

\begin{lemma}[$\PA$] \label{stepdown_exp}$\mbox{\cite[Lemma~2.10]{FriRatWei13}}$  Suppose $\alpha,\beta< \varepsilon_0$, $n$ is a number, $\mbox{$\omega^{\alpha}\stepdown{n}0$}$, and $\alpha\stepdown{n}\beta$. Then $\omega^{\alpha}\stepdown{n}\omega^{\beta}$. 
\end{lemma}

\begin{lemma}[$\PA$] \label{stepdown_aux1}  For all numbers $k$, $n$, and $s$ if $\omega_n^{k+1}\stepdown{s} 0$ then $\omega_n^{k+1}\stepdown{s} \omega_n^k$.  
\end{lemma}

\begin{lemma}[$\PA$] \label{stepdown_aux2} For all numbers $k$, $n$, and $s\ge 1$ if $\omega_n^{k+1}\stepdown{s} 0$ then $\mbox{$\omega_n^{k+1}\stepdown{s} \omega_n^k+1$}$.  
\end{lemma}
\begin{proof}
 By the Lemma \ref{stepdown_aux1}, we have $\omega^{k+1}_n\stepdown{s}\omega_n^{k}$. We show that on the step before $\omega_n^{k}$ on the $\stepdown{s}$-path from $\omega^{k+1}_n$ to $\omega^{k}_n$, we have $\omega^k_n+1$ and hence the lemma holds.  Case consideration shows that there are at most two possible $\alpha$-s such that $\alpha \stepdown[1]{s} \omega_n^k$: the ordinal $\alpha=\omega_n^k+1$ and  the ordinal $\alpha=\omega_{n+1}$, if $s+1=k$. But because $\omega_{n+1}>\omega_n^{k+1}$, any $\stepdown[1]{s}\,\,$-chain from $\omega_n^{k+1}$ to $\omega_n^k$ should go through $\omega_n^k+1$. Hence $\omega^{k+1}_n\stepdown{s} \omega_n^k+1$.
\end{proof}

\begin{lemma}[$\PA$] \label{stepdown_aux3} For all numbers $k$, $n$, and $m\le n$ if $\omega_n^{k+1}\stepdown{k} 0$   then $\mbox{$\omega_n^{k+1}\stepdown{k} \omega_m^{k+1}$}$.  
\end{lemma}
\begin{proof}
From the definition of $\stepdown{k}$ it follows that $k+1\stepdown{k} 1$. Thus from Lemma \ref{stepdown_exp} it follows that $\omega^{k+1}\stepdown{k} \omega$. We have $\omega\stepdown{k}k+1$. Hence $\omega^{k+1}\stepdown{k}k+1$. Now we use the latter and Lemma \ref{stepdown_exp} to prove the lemma by induction on $n$. 
\end{proof}

\begin{lemma}[$\PA$]\label{stepdown_aux4}  If $m\le n$, $\alpha\stepdown{m} \beta$, and $\conv{\fast_{\alpha}(n)}$ then $\conv{\fast_{\beta}(n)}$ and $\fast_{\alpha}(n)\ge\fast_{\beta}(n)$. 
\end{lemma}
\begin{proof} Using Lemma \ref{fasthier_IDOE_facts} it is sufficient to show that  $\alpha\stepdown{n} \beta$.  Consider the only sequnce $\gamma_0,\ldots,\gamma_r$ such that $\gamma_0=\alpha$, $\gamma_r=\beta$, and  $\gamma_i\stepdown[1]{m}\gamma_{i+1}$, for all $i<r$.  We show by induction on $i$ that for any $i<r$ we have $\gamma_i\stepdown{n}\gamma_{i+1}$ and $\conv{\fast_{\gamma_{i+1}}(n)}$. We first show that $\gamma_i\stepdown{n}\gamma_{i+1}$, assuming that $\conv{\fast_{\gamma_{i}}(n)}$ (we have it either from induction assumption or if $i=0$ we have it because  $\conv{\fast_{\alpha}(n)}$).  We consider two cases: $\gamma_i$ is a limit ordinal and $\gamma_i$ is a successor ordinal. The case of successor ordinal is trivial. If $\gamma_i$ is a limit ordinal then from $\conv{\fast_{\gamma_{i}}(n)}$  it follows that $\conv{\fast_{\gamma_{i}[n]}(n)}$ and thus by Lemma \ref{fasthier_IDOE_facts} we have $\gamma_i[n]\stepdown{n}0$. By \cite[Proposition~2.12]{FriRatWei13} we have $\gamma_i[n]\stepdown{n}\gamma_i[m]=\gamma_{i+1}$. Hence $\gamma_i\stepdown{n}\gamma_{i+1}$. Now from Lemma \ref{fasthier_IDOE_facts} it follows that $\conv{\fast_{\gamma_{i+1}}}(n)$. This finishes our inductive proof. Since we have $\gamma_i\stepdown{n}\gamma_{i+1}$, for any $i<r$, we clearly have $\gamma_0\stepdown{n}\gamma_r$, i.e.   $\alpha\stepdown{n} \beta$.
\end{proof}

\subsection{Transfinite induction} \label{transfinite_induction_section}

Using the representation of ordinals in $\PA$, we can formulate the schema of transfinite induction. For an ordinal $\alpha\leq\varepsilon_0$ and a number $n\geq0$, we write $\TI{\Pi_n}{\alpha}$ for the following schema:
\begin{equation}
\forall \beta<\alpha\, \left(\forall\, \gamma<\beta\, \varphi(\gamma) \to \varphi(\beta)\right)\to \forall \gamma<\alpha\, \varphi(\gamma), 
\end{equation}
where $\varphi$ is a $\Pi_n$-formula. Since there is a $\Pi_n$-partial truth definition $\fofl{True}_{\Pi_n}$ in $\IDOE$ (see Section \ref{meta_IDOE}), there is an instance of the schema that implies all other instances of it in 
$\IDOE$. We can thus identify $\TI{\Pi_n}{\alpha}$ with this instance, and consider $\TI{\Pi_n}{\alpha}$ to be a single formula.

It follows from Gentzen's work in \cite{Gen43} 
 that $\PA$ proves $\TI{\Pi_n}{\alpha}$ for all $n$ and $\alpha<\varepsilon_0$, and that it does not prove $\TI{\Pi_0}{\varepsilon_0}$. For a treatment of the amount of transfinite induction available in the fragments $\ISi{n}$ of $\PA$, see for example \cite{Som95}. 

Suppose we argue in $\PA$, and want to show that a certain property $\varphi$ holds for all ordinals less than some $\alpha<\varepsilon_0$. By the above, it suffices to show that 
\begin{equation}\label{progr}
\forall \beta<\alpha\, \left(\forall\, \gamma<\beta\, \varphi(\gamma) \to \varphi(\beta)\right). 
\end{equation}
A formula $\varphi$ for which (\ref{progr}) holds will be called \emph{progressive}. We note that for any $\alpha\leq \varepsilon_0$, $\ISi{1}$ verifies that the formula $\conv{\fast_{\alpha}}$ is progressive, i.e.\ that
\begin{equation}\label{fast_prog}
\forall \gamma< \beta\, \conv{\fast_{\gamma}}\to \conv{\fast_{\beta}}.
\end{equation}
To see that (\ref{fast_prog}) holds, note that by Definition \ref{fasthier}, the following are verifiable in $\ISi{1}$: 
\begin{enumerate}
\item $\conv{\fast_0}$
\item $\forall \alpha \left(\conv{\fast_{\alpha}}\to \conv{\fast_{\alpha+1}}\right)$
\item $\forall \lambda\leq \varepsilon_0\, \left(\lambda\in \mathit{Lim}\to (\forall \alpha<\lambda\,  \conv{\fast_{\alpha}}\to \conv{\fast_{\lambda}})\right)$
\end{enumerate}
Thus whether a function $\fast_{\alpha}$ (for some $\alpha\leq \varepsilon_0$) is provably total in some extension $\thr{T}$ of $\ISi{n}$ depends on the amount of transfinite induction available in $\thr{T}$.



\section{Transfinite iterations of provability predicates}\label{Transfinite_Section}
In the present section we will give precise definitions of  transfinite iterations of provability predicates and their duals. These notions are closely related to Turing-Feferman progressions \cite{Tur39,Fef62}.  Our presentation of this subject is based on the approach from \cite{Bek03} which itself is based on  \cite{Sch79}.

\begin{definition}We say that $(D,\prec)$ is an elementary linear ordering if $D$ is a subset of the natural numbers, for both $D$ and $\prec$ there are fixed bounded formulas of the language of $\thr{EA}$ that define them, and $\thr{EA}$ proves that $(D,\prec)$  is a linear ordering.
\end{definition}

Note, that for any elementary well-ordering $(D,\prec)$ there are $\Sigma_1$ formulas of the language of first-order arihmetic $\mathcal{L}$ that  are equivalent in $\thr{EA}$ to the standard defining formulas for $D$ and $\prec$. Because $\thr{EA}$ is a conservative extension of $\IDOE$, the choice of the formulas above is unique up to $\IDOE$-provable equivalence. Thus we can freely talk about provability of facts about an elementary well ordering within theories containing $\IDOE$.

We will define transfinite iterations of  provability predicates. Reflexive induction is an important method of reasoning about such iterations.
\begin{lemma}$\mbox{\cite[Lemma~2.4]{Bek03}\cite{Sch79}}$ \label{reflexive_induction} For any elementary linear ordering $\mbox{$(D,\prec)$}$, any theory $\thr{T}$ extending $\IDOE$ is closed under the following reflexive induction rule:
$$\frac{\forall \alpha \in D\, (\Box_{\thr{T}}\forall \beta\prec \alpha\, \fofl{F}(\beta)\to\fofl{F}(\alpha))}{\forall \alpha\in D\, \fofl{F}(\alpha)}.$$ 
\end{lemma}
\begin{proof} Suppose  $\thr{T}\vdash \forall \alpha \in D\, ((\Box_{\thr{T}}\forall \beta\prec \alpha\, \fofl{F}(\beta))\to\fofl{F}(\alpha))$. Then the sentence with stronger assumption is also derivable: $$\thr{T}\vdash \forall \alpha \in D\, ((\Box_{\thr{T}}\forall \beta \in D\,\fofl{F}(\beta))\to\fofl{F}(\alpha)).$$ We can also weaken the conclusion: $$\thr{T}\vdash \Box_{\thr{T}}\forall \alpha \in D \,\fofl{F}(\alpha)\to \forall \alpha \in D \,\fofl{F}(\alpha).$$ Therefore by Löb's theorem we have $$\thr{T}\vdash \forall \alpha\in D\,\fofl{F}(\alpha).$$
\vspace{-38pt}

\end{proof}

We fix  for the rest of the section  a $\Sigma_1$-provability predicate $\triangle$ for an arithmetical theory $\thr{T}$ containing $\IDOE$ that satisfies Hilbert-Bernays-Löb derivability conditions verifiably in $\IDOE$.  We denote by $\triangleDual$ the dual consistency predicate for $\triangle$. Also, we fix an elementary linear ordering  $(D,\prec)$  such that the least element of $(D,\prec)$ is $0^D$ and the fact that $0^D$ is the least element of $(D,\prec)$  is verifiable $\IDOE$. 

 We define iterations of $\triangle$ along $(D,\prec)$: $\triangle^\alpha\varphi$, where  $\alpha\in D$ and $\varphi$ is an arithmetical sentence. An iteration $\triangle^xy$ is an arithmetical formula with two free variables such that
$$\IDOE\vdash \forall \varphi \forall \alpha\in D\setminus \{0^D\}\, (\triangle^{\alpha} \varphi\mathrel{\leftrightarrow}\exists \beta\mathrel{\prec} \alpha \triangle\triangle^\beta \varphi).$$
Here and below, if we refer to Gödel numbers of iterations, we could also use zero times iterations. We define $\triangle^{0^D}\varphi$ to be equal to $\varphi$, i.e.\ more formally, $\triangle\triangle^\beta \varphi$ should be written as $$(\beta=0^D\to \triangle \varphi)\land (\beta\ne 0^D\to \triangle \triangle^{\beta}\varphi).$$

Existence of iterations follows from the Diagonal Lemma (Theorem \ref{diagonal_lemma}). Simple inspection of the last argument shows that the resulting formula is $\Sigma_1$. Actually any two iterations are $\IDOE$-provably equivalent (this fact resembles uniqueness of smooth progressions \cite{Bek03}).  

\begin{lemma} \label{prov_iter_unique} For any two iterations $(\triangle^xy)_1$ and $(\triangle^xy)_2$ of $\triangle$ along $(D,\prec)$ we have
\begin{equation}\label{triangle_equiv_eq_1}\IDOE\vdash \forall \alpha\in D\setminus \{0^D\}\,\forall \varphi ((\triangle^\alpha\varphi)_1\mathrel{\leftrightarrow}(\triangle^\alpha\varphi)_2)).\end{equation}
\end{lemma}
\begin{proof}
We use reflexive induction to prove it. We need to show that
\begin{equation}\begin{aligned}\IDOE\vdash \forall \alpha\in D\setminus \{0^D\} \,(& \Box_0\forall \beta\prec \alpha \forall \varphi ((\triangle^\beta\varphi)_1\mathrel{\leftrightarrow}(\triangle^\beta\varphi)_2)\to \\ & \forall \varphi ((\triangle^\alpha\varphi)_1\mathrel{\leftrightarrow}(\triangle^\alpha\varphi)_2)).\end{aligned}\end{equation}
By definition of an iteration the latter will follow from 
\begin{equation}\begin{aligned}\IDOE\vdash \forall \alpha\in D\setminus \{0^D\} \,( & \Box_0\forall \beta\prec \alpha \forall \varphi ((\triangle^\beta\varphi)_1\mathrel{\leftrightarrow}(\triangle^\beta\varphi)_2)\to \\ &  \forall \varphi (\exists \beta \mathrel{\prec}\alpha \triangle ( \triangle^\beta\varphi)_1\mathrel{\leftrightarrow}\exists \beta \mathrel{\prec}\alpha \triangle(\triangle^\beta\varphi)_2)).\end{aligned}\end{equation}
Because there is a symmetry between $(\triangle^xy)_1$ and $(\triangle^xy)_2$, it is enough to show that
\begin{equation}\label{triangle_equiv_eq_3}\begin{aligned}\IDOE\vdash \forall \alpha\in D\setminus \{0^D\} \,( & \Box_0\forall \beta\prec \alpha \forall \varphi ((\triangle^\beta\varphi)_1\to(\triangle^\beta\varphi)_2) \to \\ &  \forall \varphi (\exists \beta \mathrel{\prec}\alpha \triangle ( \triangle^\beta\varphi)_1\to\exists \beta \mathrel{\prec}\alpha \triangle(\triangle^\beta\varphi)_2)).\end{aligned}\end{equation}
Clearly, we have 
\begin{equation}\begin{aligned}\IDOE\vdash \forall \alpha\in D\setminus \{0^D\} \,(  \Box_0\forall &\beta\prec \alpha \forall \varphi ((\triangle^\beta\varphi)_1\to(\triangle^\beta\varphi)_2) \to \\ &  \forall \beta\prec \alpha \forall \varphi \Box_0((\triangle^\beta\varphi)_1 \to (\triangle^\beta\varphi)_2)).\end{aligned}\end{equation}
Because $\thr{T}$ contains $\IDOE$ we have 
$$\begin{aligned} \IDOE\vdash \forall \varphi \forall \beta\in D\setminus \{0^D\} (\Box_0 & ((\triangle^\beta\varphi)_1 \to (\triangle^\beta\varphi)_2)\\ & \to \triangle((\triangle^\beta\varphi)_1\to (\triangle^\beta\varphi)_2)).\end{aligned}$$
Thus
$$\begin{aligned}\IDOE\vdash  \forall \varphi \forall \beta\in D\setminus \{0^D\} (\Box_0 & ((\triangle^\beta\varphi)_1 \to (\triangle^\beta\varphi)_2)\\ &\to (\triangle(\triangle^\beta\varphi)_1\to \triangle (\triangle^\beta\varphi)_2)).\end{aligned}$$
Hence (\ref{triangle_equiv_eq_3}) holds and we have (\ref{triangle_equiv_eq_1}).
\end{proof}

In the same fashion as iterations of $\triangle$ we define the dual notion of  iterations $\triangleDual^xy$ of $\triangleDual$.  $\triangleDual^xy$  is an arithmetical formula with two free variables such that
$$\IDOE\vdash \forall \varphi \forall \alpha\in D\setminus \{0^D\}\,(\triangleDual^{\alpha} \varphi\mathrel{\leftrightarrow}\forall \beta\mathrel{\prec} \alpha (\triangleDual\triangleDual^\beta \varphi)).$$

Existence of iterations $\triangleDual^xy$  again follows from Diagonal Lemma.  
\begin{lemma} \label{consistency_iter_unique} For any two iterations $(\triangleDual^xy)_1$ and $(\triangleDual^xy)_2$ of $\triangleDual$ along $(D,\prec)$ we have
\begin{equation}\label{triangledown_equiv_eq_1}\IDOE\vdash \forall \alpha\in D\setminus \{0^D\}\,\forall \varphi ((\triangleDual^\alpha\varphi)_1\mathrel{\leftrightarrow}(\triangleDual^\alpha\varphi)_2)).\end{equation}
\end{lemma} 
\begin{proof} Can be proved in the same fashion as Lemma \ref{prov_iter_unique}. \end{proof}

Because we have existence and uniqueness (up to provable equivalence), we use iterations $\triangle^\alpha\varphi$ and $\triangleDual^\alpha\varphi$ freely, without specifing explicit formulas.

Let us denote by $\fofl{Succ}_D(\alpha,\beta)$ the formula $$\alpha\in D\land \beta\in D\land \alpha\prec\beta\land\forall \gamma\in D \,\lnot(\alpha\prec \gamma \land \gamma\prec \beta).$$Let us denote by $D^{\mathrm{lim}}$ the set of all  $\alpha\in D$ such that $$\alpha\ne0^D\land \forall \beta\in D\,\lnot \fofl{Succ}_D(\beta,\alpha).$$

\begin{lemma}[$\IDOE$]\label{iteration_facts_1}The following properties of iterations of $\triangle$ and $\triangleDual$ hold:
\begin{enumerate}
\item \label{iteration_facts_1_1}$\forall \varphi\forall  \alpha\in D\setminus \{0^D\}\, (\lnot\triangle^\alpha\lnot\varphi\mathrel{\leftrightarrow} \triangleDual^{\alpha}\varphi)$;
\item \label{iteration_facts_1_2}$\forall \varphi\forall  \alpha\in D\setminus \{0^D\}\,(\triangle\varphi \to  \triangle^\alpha\varphi )$;
\item \label{iteration_facts_1_3}$\forall \varphi\forall  \alpha,\beta\in D\setminus \{0^D\}\,(\alpha\prec \beta \to (\triangle^\alpha\varphi \to  \triangle^\beta\varphi ))$;
\item \label{iteration_facts_1_4}$\forall \alpha \in D\setminus \{0^D\} \,\forall \varphi,\psi(\triangle^{\alpha}(\varphi\to\psi)\to(\triangle^{\alpha}\varphi \to \triangle^{\alpha}\psi))$;
\item \label{iteration_facts_1_5}$\forall \alpha \in D\setminus \{0^D\}\,\forall \varphi \in \Sigma_1\text{-Sen}(\fofl{True}_{\Sigma_1}(\varphi)\to  \triangle^{\alpha}\varphi)$;
\item \label{iteration_facts_1_6}$\forall \alpha \in D\setminus \{0^D\}\,\forall \varphi (\triangle^{\alpha}\varphi\to \triangle^{\alpha}\triangle^{\alpha}\varphi)$;
\item  \label{iteration_facts_1_7}$\forall \alpha\in D\setminus \{0^D\}\,\forall \beta(\fofl{Succ_D}(\beta,\alpha)\to \forall \varphi(\triangle^{\alpha}\varphi\mathrel{\leftrightarrow} \triangle\triangle^{\beta}\varphi))$;
\item  \label{iteration_facts_1_8}$\forall \alpha\in D^{\mathrm{lim}}\,\forall \varphi(\triangle^{\alpha}\varphi\mathrel{\leftrightarrow}\exists \beta\prec \alpha(\beta\in D\setminus\{0^D\} \land \triangle^{\beta}\varphi))$.

\end{enumerate}
\end{lemma}
\begin{proof} The proof of items \ref{iteration_facts_1_1}, \ref{iteration_facts_1_2}, \ref{iteration_facts_1_3}, \ref{iteration_facts_1_4}, and \ref{iteration_facts_1_5} is straightforward by using reflexive induction. Item \ref{iteration_facts_1_6} follows from item \ref{iteration_facts_1_5}.   We prove item \ref{iteration_facts_1_7} using item \ref{iteration_facts_1_3} and we prove item \ref{iteration_facts_1_8} using items \ref{iteration_facts_1_7} and \ref{iteration_facts_1_3}.
\end{proof}

Lemma \ref{iteration_facts_1} gives us a number of facts about iterations of provability predicates and their duals. We will frequently use them below without explicitely refering to them.

Note that items \ref{iteration_facts_1_2}, \ref{iteration_facts_1_4}, \ref{iteration_facts_1_6} of Lemma \ref{iteration_facts_1}  yield $\IDOE$-verifiable Hilbert-Bernays-Löb derivability conditions for each $\triangle^{\alpha}$. The latter with Lemma \ref{iteration_facts_1} item \ref{iteration_facts_1_1} means that  each $\triangleDual^{\alpha}$ is dual for the provability predicate $\triangle^{\alpha}$.

\begin{lemma}  \label{iterations_monotonicity_for_predicates}Suppose $\thr{T}$ is an arithmetical theory such that $\IDOE\subseteq \thr{T}$, $\triangle_1$ and $\triangle_2$ are $\Sigma_1$ provability predicates that satisfy $\IDOE$-verifiable Hilbert-Bernays-Löb derivability conditions, and $\thr{T}\vdash \forall \varphi \,(\triangle_1\varphi\to \triangle_2\varphi)$. Then $$\thr{T}\vdash\forall \alpha\in D\setminus \{0^D\} \,\forall \varphi \,(\triangle_1^{\alpha}\varphi\to \triangle_2^{\alpha}\varphi).$$
\end{lemma}
\begin{proof} By reflexive induction.\end{proof}

For the rest of the section we assume that $(D,\prec)$ is Cantor ordinal notations for the ordinals $\le\varepsilon_0$ as defined in Section \ref{Ordinal_Section}.
\begin{lemma}[$\IDOE$] \label{iterations_and_addition} $\forall  \alpha\ge 1\forall \beta\le \varepsilon_0\forall \varphi\,(\triangle^{\alpha}\triangle^{\beta}\varphi \mathrel{\leftrightarrow} \triangle^{\beta+\alpha}\varphi)$ 
\end{lemma}
\begin{proof} By reflexive induction on $\alpha$.\end{proof}


\section{Slow provability}\label{slowprovability}
Suppose that the theory $\thr{S}$ is given by a uniform r.e.\ enumeration $\{{\thr{S}}_n\}_{n\in\omega}$. We can use any (partial) recursive 
function $\recfun{f}$ to ``slow down'' provability in $\thr{S}$, by considering the theory  
 \begin{equation}\label{Sf}
\thr{S}{\upharpoonright}_{\recfun{f}}:=\IDOE \cup  \bigcup \{\thr{S}_n\mid \conv{\recfun{f}(n)}\}.
\end{equation}
The reason for adding $\IDOE$ is to ensure that all our theories exhibit a minimal amount of nice behaviour (see Section \ref{ArTheories_Section}).
From the definition, it is clear that $\thr{S}{\upharpoonright}_{\recfun{f}}$ is a r.e.\ subtheory of $\thr{S}$.
If $\recfun{f}$ is total, then $\thr{S}{\upharpoonright}_{\recfun{f}}$ has exactly the same theorems as $\thr{S}$. However this fact may not be verifiable in a theory where $\recfun{f}$ is not provably total.

Since $\{{\thr{S}}_n\}_{n\in\omega}$ is uniformly r.e.,  there is 
an arithmetical formula $\Box_{\thr{S}_x}$, containing $x$ as a free variable, such that 
$\Box_{\thr{S}_n}$ 
is the provability predicate of $\thr{S}_n$. We assume that, verifiably in $\IDOE$, $\Box_{\thr{S}}\varphi$
is provably equivalent to
 $\exists x\, \Box_{\thr{S}_x}\varphi$. 
The provability predicate 
$\Box_{\thr{S}_x, \recfun{f}}$ of $\thr{S}{\upharpoonright}_{\recfun{f}}$ can be defined
in a natural way:

\begin{definition}\label{predseqf}
The provability predicate 
$\Box_{\thr{S}_x, \recfun{f}}$ of $\thr{S}{\upharpoonright}_{\recfun{f}}$ is defined as
\begin{equation*}
\Box_{\thr{S}_x, \recfun{f}}\,\varphi:= \Box_0\varphi \lor \exists y\leq x\, \left( \Box_{\thr{S}_y} \varphi \land \conv{\recfun{f}(x)}\right).
\end{equation*}
If $\{\thr{S}_n\}_{n\in\omega}$ is $\{\ISi{n}\}_{n\in\omega}$, we write $\Box_{\recfun{f}}$ instead of $\Box_{\thr{S}_x, \recfun{f}}$.
\end{definition}

If $\recfun{f}$ is total but not provably total in $\thr{T}$, then from the point of view of $\thr{T}$ the formula $\Box_{\thr{S}_x, \recfun{f}}$ is a \emph{nonstandard} provability predicate for $\thr{S}$. On the other hand, $\Box_{\thr{S}_x, \recfun{f}}$ is a \emph{standard} $\Sigma_1$-provability predicate for the r.e.\ theory $\thr{S}{\upharpoonright}_{\recfun{f}}$. It therefore satisfies the Hilbert-Bernays-L\"ob derivability conditions verifiably in $\IDOE$ (Section \ref{meta_IDOE}).

%

With Definition \ref{predseqf}, the usual provability predicate for $\PA$ can be written as 
$\Box_{\recfun{f}}$, where $\recfun{f}$ is any Kalmar elementary function (assuming $\IDOE$ as our metatheory). 
It is easy to see that for any $\recfun{f}$, $\Box_{\recfun{f}}\varphi$ is 
provably equivalent in $\IDOE$ to the formula: 
\begin{equation}
\exists x\, (\Box_x\varphi \land \conv{\recfun{f}(n)}).
\end{equation}
We recall the slow provability predicate studied in \cite{FriRatWei13}, defined as: 
\begin{equation}\label{frw}
\exists x\, \left(\Box_x \varphi \land \conv{\fast_{\varepsilon_0}(x)}\right).
\end{equation}
We define, for $z\in \mathbb{Z}$,
 \begin{equation}
 \fast_{\varepsilon_0}^{(z)}(x):= \fast_{\varepsilon_0}(x\dotminus z),
 \end{equation}
and consider the provability predicates $\Box_{\fast_{\varepsilon_0}^{(z)}}$. For the sake of readability, we let 
\begin{equation}
\BoxSl_z \, \varphi:=\Box_{\fast_{\varepsilon_0}^{(z)}}\varphi.
\end{equation}
Thus the provability predicate in (\ref{frw}) becomes 
$\BoxSl_0$.

\begin{remark}\label{shifting}
For any $z$, we can define a ``shifted" enumeration $\{\thr{T}_n^{z}\}_{n\in \omega}$ of $\PA$, such that  
$\BoxSl_z$ is provably equivalent to 
$\Box_{\thr{T}_{x}^{z}, \fast_{\varepsilon_0}}$, by simply defining $\thr{T}^{z}_x$ as $\ISi{x+z}$.
\end{remark}

%

In \cite{FriRatWei13}, Theorem \ref{slowcon} below is proven for
$\BoxSl_0$. 
In order to consider the more general case, we need one more definition.

\begin{definition}
We say that $\{{\thr{S}}_n\}_{n\in\omega}$ is a \emph{recursive sequence of finitely axiomatizable theories} if 
there is a recursive sequence $\{\thr{S}_n\mbox{-}\fofl{Ax}\}_{n\in\omega}$ such that for all $n$,
$\thr{S}_n$ is axiomatized by $\thr{S}_n\mbox{-}\fofl{Ax}$, and 
$\thr{S}_n\mbox{-}\fofl{Ax}$ is finite.
\end{definition}

\begin{theorem} \label{slowcon}  Suppose  $\thr{S}_n$ is a  recursive sequence of finitely axiomatizable theories such that $\PA$ proves that $\PA=\bigcup \limits_{n\in \omega} \thr{S}_n$. Let $\triangle$ denote the provability predicate $\Box_{\thr{S}_x, \fast_{\varepsilon_0}}$. Then
 $\PA\vdash \forall \varphi\left( \Box\triangle \varphi \to \Box \varphi\right)$. 
\end{theorem}

\begin{proof} Essentially the same as  \cite[Theorem~4.1]{FriRatWei13}. See also Theorem \ref{Epsilon_note_root_theorem} in Section \ref{Models_Section} below. 
\end{proof}
It follows that, from the point of view of $\PA$, any provability predicate $\triangle$ as in the statement of Theorem \ref{slowcon} defines in fact a weaker theory than $\Box$:

\begin{corollary}\label{between}
Suppose  $\thr{S}_n$  is a  recursive sequence of finitely axiomatizable theories such that $\PA$ proves that $\PA=\bigcup \limits_{n\in \omega} \thr{S}_n$. Let $\triangle$ denote the provability predicate $\Box_{\thr{S}_x, \fast_{\varepsilon_0}}$. Then
$\PA \nvdash\Box \bot \to\triangle\bot $. 
\end{corollary}

\begin{proof}
Suppose that $\PA\vdash\Box \bot \to\triangle\bot $. 
Since $\triangle \varphi $ implies $\Box\varphi$ for all $\varphi$, we have
\begin{equation}
\PA\vdash\triangle\triangle\bot \to \square\triangle\bot, 
\end{equation}
whence by Theorem \ref{slowcon}, 
\begin{equation}
\PA\vdash \triangle\triangle\bot  \to \Box \bot.
\end{equation}
Combining this with our assumption yields 
$\PA\vdash\triangle\triangle\bot \to\triangle\bot$. 
By L\"ob's Theorem for $\triangle$ (this follows from the Hilbert-Bernays-L\"ob derivability conditions for $\triangle$), we now have that 
$\PA\vdash\triangle\bot$ whence also $\PA\vdash \Box\bot$, contradiction.
\end{proof}

Theorem \ref{slowcon} holds for a rather wide class of provability predicates $\triangle$. In Section \ref{ProvLogic_Section} below, we determine the joint provability logic of any such $\triangle$ and ordinary $\PA$-provability. In contrast, the following sections provide examples of properties where the exact axiomatization of slow provability leads to a 
radical difference in the behaviour of the corresponding provability predicates. 
In particular, we show that 
\begin{equation}
\PA \vdash   \BoxSl_1^{\varepsilon_0}\,\varphi \leftrightarrow \Box \varphi, 
\end{equation}
whereas 
\begin{equation}
\PA \vdash   \BoxSl_{2}^{\omega}\,\varphi \leftrightarrow \Box \varphi, 
\end{equation}
where $\BoxSl_z^{\alpha}$ denotes the $\alpha$-iteration of $\BoxSl_z$ (see Section \ref{Transfinite_Section}). 
In Section \ref{Sqrt_Section}, we shall furthermore show that there is a function $\recfun{r}$ such that \begin{equation}
\PA \vdash \Box_{\recfun{r}}\Box_{\recfun{r}}\varphi \leftrightarrow \Box \varphi.
\end{equation}
Thus the slow provability predicates $ \BoxSl_1$, $\BoxSl_{2}$, and $\Box_{\recfun{r}}$ may be seen as the $\varepsilon_0$-root, the $\omega$-root, and the 
square root of ordinary provability, respectively.


\section{Provability implies iterated slow provability}\label{Converting_Section}
In the section we will show that in some cases provability of an arithmetical sentence in  $\PA$ implies provability of the same sentence with respect to certain transfinite iterations of some slow provability predicates.
\begin{lemma}[$\IDOE$]\label{Box_n_prov_conv} For all $n$ and $k$ we have $\Box_n\conv{\fast_{\omega_n}(k)}$.
\end{lemma}
\begin{proof}
In the case of $n=0$ the statment of lemma holds because $\fast_{\omega_0}(x)=\fast_1(x)=2x+1$.  In the rest of the proof we consider the case of $n\ge 1$.

Note that we have $\Box_n(\conv{\fast_{\omega_n}(k)}\mathrel{\leftrightarrow}\conv{\fast_{\omega_{n-1}^{k+1}}(k)})$. Thus it is enough to show that for every $\alpha<\omega_n$  we have $\Box_n\conv{\fast_{\alpha}}$.  Since the theory $\ISi{1}$ proves that $\conv{\fast_{\alpha}}$ is a progressive formula (see Section \ref{transfinite_induction_section}) and $\conv{\fast_{\alpha}}$ is a $\Pi_2$-formula, we have $\Box_1\forall \alpha\le \varepsilon_0\,(\TI{\Pi_2}{\alpha}\to\conv{\fast_{\alpha}})$. Hence it is enough to show that for all $\alpha<\omega_n$ we have $\Box_n\TI{\Pi_2}{\alpha}$.

If $n=1$ then we only need to show that for any number $m$ we have $\Box_n\TI{\Pi_2}{m}$. Clearly, for every $m$ we have $\IDOE\vdash \TI{\Pi_2}{m}\to \TI{\Pi_2}{m+1}$. Thefore for any $m$ we have $\IDOE\vdash \TI{\Pi_2}{m}$ and hence $\Box_n\TI{\Pi_2}{m}$.  Thus we will consider only the case of $n\ge 2$.

  In \cite[Theorem~4.1]{Som95} it has been shown that if $0<m\le n$ and $\omega\le\alpha<\varepsilon_0$ then $\IDO$ proves that $\TI{\Pi_n}{\alpha}$ implies  $\TI{\Pi_m}{\beta}$ for all $\beta<\omega_{n-m}^{\alpha^{\omega}}$; a simple inspection of the proof shows that that the argument could be formalized in $\IDOE$. Thus for all $\alpha<\omega_n$ we have $\Box_1(\TI{\Pi_{n}}{\omega}\to\TI{\Pi_2}{\alpha})$. But $\TI{\Pi_{n}}{\omega}$ is just $\Pi_n$-induction for natural numbers and is well-known to be equivalent in $\IDO$  to $\Sigma_n$-induction \cite[Theorem~I.2.4]{HajPud98}. Therefore  for all $\alpha<\omega_n$ we have $\Box_n\TI{\Pi_2}{\alpha}$.
\end{proof}

\begin{corollary} \label{slowcon_inv} Suppose  $\thr{S}_n$ is a recursive sequence of finitely axiomatizable theories such that $\PA$ proves that $\PA=\bigcup \limits_{n\in \omega} \thr{S}_n$.  Let $\triangle$ denote the provability predicate $\Box_{\thr{S}_x, \fast_{\varepsilon_0}}$. Then
 $\PA\vdash \forall \varphi\left( \Box \varphi \to \Box\triangle \varphi \right)$. 
\end{corollary}

\begin{lemma}[$\ISi{1}$] \label{omega_upper_bound} Suppose $\varphi$ is an arithmetical sentence and we have $\Box \varphi$. Then there exists a number $n$ such that $\BoxSl^n_2\varphi$.
\end{lemma}
\begin{proof} Because we have $\Box \varphi$ we also have $\Box_n \varphi$ for some $n$. If $n< 2$ then because $\conv{\fast_{\varepsilon_0}(1)}$, we have $\BoxSl^1_2\varphi$. Thus it remains to consider the case of $n\ge 2$. 

Now we prove by induction on $m>0$  that $\BoxSl_2^m\conv{\fast_{\varepsilon_0}(m-1)}$. 
 We have $\conv{\fast_{\varepsilon_0}(0)}$ and by $\Sigma_1$-completeness of $\BoxSl_2$ we have $\BoxSl_2\conv{\fast_{\varepsilon_0}(0)}$. Thus induction base holds. We recall that  $\BoxSl_{2}^m$ satisfies Hilbert-Bernays-Löb conditions. From $\BoxSl_{2}^m\conv{\fast_{\varepsilon_0}(m-1)}$  it follows that $\BoxSl_{2}^m \forall \psi (\Box_{m+1}\psi\to \BoxSl_2\psi)$. From Lemma \ref{Box_n_prov_conv}  we have $\Box_{m+1}\conv{\fast_{\omega_{m+1}}(m)}$, hence we have $\Box_{m+1}\conv{\fast_{\varepsilon_0}(m)}$, and thus by $\Sigma_1$-completeness of $\BoxSl_2^m$ we have $\BoxSl_2^m\Box_{m+1}\conv{\fast_{\varepsilon_0}(m)}$.  Therefore we have $\BoxSl_2^m(\forall\psi (\Box_{m+1}\psi\to \BoxSl_2\psi)\land \Box_{m}\conv{\fast_{\varepsilon_0}(m)})$ and hence $\BoxSl_2^{m}\BoxSl_2\conv{\fast_{\varepsilon_0}(m)}$. Therefore the induction step holds.

Hence we have $\BoxSl_2^{n-1}\conv{\fast_{\varepsilon_0}(n-2)}$. Thus we have $\BoxSl_2^{n-1}\forall \psi (\Box_{n}\psi\to \BoxSl_2\psi)$. From $\Box_n\varphi$ it follows that $\BoxSl_2^{n-1}\Box_n\varphi$. Therefore we have $\BoxSl_2^{n}\varphi$.
\end{proof}

Combining Theorem \ref{slowcon} and Lemma \ref{omega_upper_bound} we conclude that the following theorem holds:
\begin{theorem}[$\PA$]\label{omegait}
Suppose $\varphi$ is an arithmetical sentence. Then $\Box\varphi$ iff $\BoxSl_2^{\omega}\varphi$.
\end{theorem}

We clearly have the following generalization:

\begin{theorem} \label{omega_root_theorem}Suppose  $\thr{S}_n$ is a recursive sequence of finitely axiomatizable theories such that $\PA$ proves that $\PA=\bigcup \limits_{n\in \omega} \thr{S}_n$  and $\ISi{n + 2}\subseteq \thr{S}_n$. Let $\triangle$ denote the provability predicate $\Box_{\thr{S}_x, \fast_{\varepsilon_0}}$.  Then $\PA$ proves that for every arithmetical sentence $\varphi$ we have $\Box\varphi$ iff $\triangle^{\omega}\varphi$.
\end{theorem}

\begin{lemma}[$\ISi{1}$] \label{fasthier_prov_conv_on_iter}Suppose $\alpha<\varepsilon_0$ and $n$ is a number. Then $\Box^{\omega^{\alpha}}_1\conv{\fast_{\alpha}(n)}$.
\end{lemma}
\begin{proof} We use reflexive induction and hence it is sufficient to show (while reasoning in $\ISi{1}$) that  for every $\alpha\le\varepsilon_0$, if $\Box_1\forall \beta<\alpha\, \forall x \, \Box_1^{\omega^\beta}\conv{\fast_{\beta}(x)}$ then for every $m$ we have $\Box_1^{\omega^\alpha}\conv{\fast_{\alpha}(m)}$.

We reason in $\ISi{1}$. Let us consider three cases: $\alpha=0$, $\alpha$ is a limit ordinal, and $\alpha$ is a successor ordinal. In the first case we need to show that $\Box_1\conv{\fast_0(n)}$ which is clearly true. In the second case $\Box_1^{\omega^{\alpha}}\conv{\fast_\alpha(m)}$ is equivalent to $\Box_1^{\omega^{\alpha}}\conv{\fast_{\alpha[m]}(m)}$. The latter follows from $\Box_1\Box_1^{\omega^{\alpha[m]}}\conv{\fast_{\alpha[m]}(m)}$ which itself follows from $\Box_1\forall \beta<\alpha \forall x \Box_1^{\omega^\beta}\conv{\fast_{\beta}(x)}$. 

Now let us consider the case of successor $\alpha=\gamma+1$. In order to show that $\Box_1^{\omega^{\gamma+1}}\conv{\fast_{\gamma+1}(m)}$ it is enough to show that $\Box_1\Box_1^{\omega^{\gamma}\cdot(m+1)}\conv{\fast_{\gamma}^{m+1}(m)}$. In order to prove latter we prove by induction on $k\ge 1$ that  $\Box_1\forall x \Box_1^{\omega^{\gamma}\cdot k}\conv{\fast_{\gamma}^{k}(x)}$ . The base of induction follows directly from reflexive induction assumption.  Suppose we have  $\Box_1\forall x \Box_1^{\omega^{\gamma}\cdot k}\conv{\fast_{\gamma}^{k}(x)}$. Then we have $\Box_1\Box_1^{\omega^{\gamma}}\forall x \Box_1^{\omega^{\gamma}\cdot k}\conv{\fast_{\gamma}^{k}(x)}$. Using reflexive induction assumption we obtain $$\Box_1(\,\forall x \Box_1^{\omega^{\gamma}}\conv{\fast_{\gamma}(x)} \,\,\, \land \,\,\, \Box_1^{\omega^{\gamma}}\forall x \Box_1^{\omega^{\gamma}\cdot k}\conv{\fast_{\gamma}^{k}(x)}\,).$$ Thus we have $$\Box_1\forall x \Box_1^{\omega^{\gamma}}(\, \conv{\fast_{\gamma}(x)} \,\,\, \land \,\,\, \forall y \Box_1^{\omega^{\gamma}\cdot k}\conv{\fast_{\gamma}^{k}(y)}\,).$$ Hence we have $$\Box_1\forall x \Box_1^{\omega^{\gamma}}(\exists y ( y=\fast_{\gamma}(x) \land  \Box_1^{\omega^{\gamma}\cdot k}\conv{\fast_{\gamma}^{k}(y)})).$$ Therefore $$\Box_1\forall x \Box_1^{\omega^{\gamma}}\Box_1^{\omega^{\gamma}\cdot k}\conv{\fast_{\gamma}^{k+1}(x)}.$$ This finishes the proof in the successor case.
\end{proof}

\begin{lemma} \label{epsilon_0_upper_bound} Suppose  $\thr{S}_n$ is a recursive sequence of finitely axiomatizable theories such that $\PA$ proves that $\PA=\bigcup \limits_{n\in \omega} \thr{S}_n$. Let $\triangle$ denote the provability predicate $\Box_{\thr{S}_x, \fast_{\varepsilon_0}}$. Then $\PA$ proves that for every arithmetical sentence $\varphi$ such that $\Box\varphi$ we have $\triangle^{\varepsilon_0}\varphi$. 
\end{lemma}
\begin{proof}
Let us reason in $\PA$. For some $n$ we have $\Box_{\thr{S}_n}\varphi$. From Lemma \ref{fasthier_prov_conv_on_iter} it follows that we have $\Box_1^{\omega_{n+2}}\conv{\fast_{\varepsilon_0}(n)}$. Since for some $n$ the theory $\ISi{1}\subseteq  \thr{S}_n$, in $\PA$  the predicate $\triangle$ is at least as strong as $\Box$. Therefore by Lemma \ref{iteration_facts_1} and Lemma \ref{iterations_monotonicity_for_predicates} we have $\triangle^{\varepsilon_0}\conv{\fast_{\varepsilon_0}(n)}$. Hence we have $\triangle^{\varepsilon_0}(\conv{\fast_{\varepsilon_0}(n)}\land\Box_{\thr{S}_n}\varphi)$. Thus $\triangle^{\varepsilon_0}\triangle\varphi$. Finally,  by Lemma \ref{iterations_and_addition} we have $\triangle^{\varepsilon_0}\varphi$. 
\end{proof}


\section{Models for slow consistency}\label{Models_Section}
In this section we will show that for the slow provability predicate $\BoxSl_1$, the theory $\PA$ proves that $\BoxSl_1^{\varepsilon_0}$ is equivalent to $\Box$.   We also show that   in addition to $\BoxSl_1$ a large family of provability predicates has the same property.



We will use model-theoretic techniques while reasoning in $\PA$.  We briefly present basic definitions of formalization of model theory within arithmetic. A \emph{model} of finite signature   is a tuple of formulas that give the domain, as well as interpretations of the constant, functional, and predicate symbols.  A \emph{full} model is a model with a satisfaction relation for all first-order sentences of the signature of the model given by their Gödel numbers. Note that for a model without satisfaction relation there is no straightforward way to formalize in $\PA$ whether the model satisfies some infinite set of axioms, hence when we will talk about models of $\PA$ within $\PA$, we will assume that the models are full models. We will formalize our model-theoretic arguments within $\PA$ and since full induction schema is present, our arguments will not depend on the complexity of formulas giving models. We also recall that Gödel Completeness theorem is formalizable in $\PA$, i.e. in $\PA$ the consistency of a theory $\thr{T}$ implies the existence of a full model of $\thr{T}$. Also, in $\PA$ the existence of some full model for a theory $\thr{T}$ implies the consistency of $\thr{T}$.  

For models of arithmetical theories  we also consider partial satisfaction relation that are defined only for $\Pi_n$ (or equivalently, $\Sigma_n$) formulas. Note that if we fix a number $k$ externally, then in $\PA$ we can construct a $\Pi_{n+k}$ partial satisfaction relation from a $\Pi_n$ partial satisfaction relation. Also note that since $\PA$ proves the Cut Elimination Theorem, it also proves that if  $\thr{T}$ is an arithmetical theory axiomatizable by a recursive set of $\Pi_n$-formulas and there is a model of $\thr{T}$ with a $\Pi_n$ partial satisfaction relation, then $\thr{T}$  is consistent.

For a proper presentation of the basic definitions and the proofs of basic theorems of model theory in formal arithmetic the reader is referred to \cite[I.4(b)]{HajPud98}.

Below we work with non-standard models of arithmetical signature, all of them will be models of $\IDOE$. Because standard natural numbers are embeddable as initial segments in every model of $\IDOE$ we freely assume that every such model contains the standard numbers. 

Suppose that $\model{M}$ is a model of $\IDOE$. A cut $\model{I}$ of $\model{M}$ is a submodel of $\model{M}$ such that for any $a\in\model{M}$ and $b\in \model{J}$, if $\model{M}\models a<b$,  then we have $a\in\model{I}$. For every cut $\model{I}$ and an element $a\in\model{M}$ we write $\model{I}<a$ if $\forall b\in\model{I}\, (\model{M}\models b<a)$ and we write $a<\model{I}$ if $a\in\model{I}$. 

\begin{theorem}\cite[Theorem~5.25]{Som95}] \label{TI_on_cut}Suppose $n\ge 1$ is a number, $\alpha<\varepsilon_0$ is an ordinal, $\model{M}$ is a model of $\IDOE$, and  $a,b,c\in \model{M}$ are non-standard numbers such that $\model{M}\models \fast_{\omega_{n-1}^{\omega^{\alpha}\cdot c}}(a)=b$. Then there exists a cut $\model{I}$ of $\model{M}$ such that $a<\model{I}<b$ and $\model{I}$ is a model of $\IDOE+\TI{\Pi_n}{\omega^{\omega^\alpha}}$.
\end{theorem}

\begin{remark}   \label{TI_on_cut_form} It is possible to formalize Theorem \ref{TI_on_cut} in $\PA$ as a theorem schema.  We assume that the model $\model{M}$ is given by some fixed tuple of arithmetical formulas, possibly with additional parameters. Then there are arithmetical formulas $\varphi(x,n,\alpha,a,b,c)$ and $\psi(y,n,\alpha,a,b,c)$  such that $\PA$ proves that if $n,\alpha,a,b,c$ satisfy the conditions of Theorem \ref{TI_on_cut} then the set of all $x$ from $\model{M}$ for which $\varphi(x,n,\alpha,a,b,c)$ holds is a cut $\model{I}$, the formula $\psi(y,n,\alpha,a,b,c)$ gives a $\Pi_{n}$ partial satisfaction relation for $\model{I}$, and the cut $\model{I}$ satisfies the conclusion of Theorem \ref{TI_on_cut}.
\end{remark}

\begin{theorem}(\cite[Theorem~5.2]{Som95}]) \label{TI_Pi_0^2_analysis}For all $n\ge 1$ and all $\alpha,\beta<\varepsilon_0$,
$$\IDOE+\TI{\Pi_n}{\omega^{\omega^\alpha}}\vdash \conv{\fast_{\beta}}\iff \beta<\omega_n^{\alpha+1}.$$
\end{theorem}

\begin{remark}\label{TI_Pi_0^2_analysis_remark}We will employ only $\Leftarrow$ direction of Theorem \ref{TI_Pi_0^2_analysis}. Examination of the proof of \cite[Lemma~5.5]{Som95} shows that $\Leftarrow$ direction of Theorem \ref{TI_Pi_0^2_analysis} can be formalized in $\IDOE$.
\end{remark}

A result close to the following theorem goes back to Paris \cite{Par80}. In the form given below the theorem can be derived from results of Beklemishev \cite[Theorem~1, Proposition~7.3, Remark~7.4]{Bek03}, Freund  has proved this theorem explicitely \cite{Fre15} for the case of $\alpha<\omega$.
\begin{theorem}\label{one_con_fast_hier_equiv} $\mbox{\cite{Bek03,Fre15}}$ $$\ISi{1}\vdash \forall \alpha\in[1,\omega](\DiamondI[1]\!_{\alpha}\top \mathrel{\leftrightarrow} \conv{\fast_{\omega_\alpha}}).$$
\end{theorem}
\begin{proof} In \cite{Fre15} it has been proved that
$$\ISi{1}\vdash \forall \alpha\in[1,\omega)(\DiamondI[1]\!_{\alpha}\top \mathrel{\leftrightarrow} \conv{\fast_{\omega_\alpha}}).$$
We need to prove that $\ISi{1}\vdash \DiamondI[1]\!_{\omega}\top \mathrel{\leftrightarrow} \conv{\fast_{\varepsilon_0}}$.
Clearly,
$$\begin{aligned}\ISi{1}\vdash  \DiamondI[1]\!_{\omega}\top & \mathrel{\leftrightarrow} & \forall x \DiamondI[1]\!_{x}\top\\ & \mathrel{\leftrightarrow} &  \forall x \conv{\fast_{\omega_x}}.\end{aligned}$$
From Lemma \ref{fasthier_IDOE_facts} it follows that $\IDOE\vdash \forall x \,\conv{\fast_{\omega_x}} \mathrel{\leftrightarrow}  \conv{\fast_{\varepsilon_0}}$. Thus indeed $$\ISi{1}\vdash \DiamondI[1]\!_{\omega}\top \mathrel{\leftrightarrow} \conv{\fast_{\varepsilon_0}}.$$
\vspace{-38pt}

\end{proof}

\begin{lemma}\label{epsilon_0_root_lower_bound}($\PA$) For every arithmetical sentence $\varphi$ if $\Diamond \varphi$ then $\DiamondSl\!_{1}^{\varepsilon_0}\varphi$. 
\end{lemma}
\begin{proof} We reason in $\PA$. In order to prove our claim we consider two cases $\BoxI[1]\bot$ and $\lnot\BoxI[1]\bot$, i.e. $\DiamondI[1]\top$.  

We start with the case of $\BoxI[1]\bot$.  From Theorem \ref{one_con_fast_hier_equiv} for the case of $\alpha=\omega$ it follows that there exists $n_0$ such that $\dive{\fast_{\varepsilon_0}(n_0)}$  and hence $\dive{\fast_{\varepsilon_0}(n)}$, for every $n\ge n_0$. Thus for all $n\ge n_0$ and arithmetical sentences $\psi$ we have $\DiamondSl_1\psi$ if $\Diamond_n\psi$. By straightforward calculations we exclude cases $n_0=0,1$. Hence $n_0\ge 2$. 

Suppose we have been given an arithmetical sentence $\varphi$ such that $\Diamond \varphi$. We need to show that $\DiamondSl\!_1^{\varepsilon_0}\varphi$ holds. By definition $\DiamondSl\!_1^{\varepsilon_0}\varphi$ iff for every $\alpha<\varepsilon_0$ we have $\DiamondSl\!_1\DiamondSl\!_1^{\alpha}\varphi$. Thus it is sufficient to show that for every $\alpha<\varepsilon_0$ we have $\Diamond\DiamondSl\!_1^{\alpha}\varphi$.

By Completeness Theorem we have a model $\model{M}$ of $\PA+\varphi$. Let us consider arbitrary $n > n_0$. We will construct a cut $\model{I}_{n}$ of $\model{M}$ such that $\model{I}_{n}$ is a model of $\IDOE+\TI{\Pi_1}{\omega_{n+1}}+\dive{\fast_{\varepsilon_0}(n)}$. 

First, assume that we have already contstructed  the cut $\model{I}_n$. Giving that $\PA$ is essentially reflexive, $\model{M}\models\Diamond\!_{n}\varphi$. Hence because $\Diamond\!_{n}\varphi$ is $\Pi_1$ we have $\model{I}_n\models \Diamond\!_{n}\varphi$.  Combining Theorem \ref{TI_Pi_0^2_analysis}, Remark \ref{TI_Pi_0^2_analysis_remark}, and Theorem \ref{one_con_fast_hier_equiv} we see that $\IDOE+\TI{\Pi_1}{\omega_{n+1}}\vdash \DiamondI[1]_{n}\top$ and thus $\model{I}_n\models\DiamondI[1]_n\top$.  

Let us reason in $\IDOE+\TI{\Pi_1}{\omega_{n+1}}+\dive{\fast_{\varepsilon_0}(n)}+\DiamondI[1]_{n}\top+\Diamond\!_{n}\varphi$. We claim that $\Diamond\!_{n}\DiamondSl\!_1^{\alpha}\varphi$, for all non-zero $\alpha<\omega_{n+1}$. We prove by transfinite induction that $\DiamondSl\!_1^{\alpha}\varphi$, for all non-zero $\alpha<\omega_{n+1}$. The base holds because $\DiamondSl\!_1\varphi$ follows from $\Diamond\!_{n}\varphi$.  The limit case follows directly from definition. For the successor case we use the fact that $\DiamondSl\!_{1}^{\alpha}\varphi$ is $\Pi_1$ and that $\ISi{n}$ is consistent with all true $\Pi_1$-sentences. Thus we obtain  $\Diamond\!_{n}\DiamondSl\!_1^{\alpha}\varphi$ and next $\DiamondSl\!_1^{\alpha+1}\varphi$. This finishes the inductive proof.  In order to prove our claim we note that $\Diamond\!_{n}\DiamondSl\!_1^{\alpha}\varphi$  follows from $\DiamondSl\!_1^{\alpha}\varphi$ since $\DiamondSl\!_1^{\alpha}\varphi$ is $\Pi_1$.
Thus if $\model{I}_n$ will have the above properties, we will have $\model{I}_n\models \Diamond\!_n\DiamondSl\!_1^{\alpha}\varphi$, for all non-zero $\alpha<\omega_{n+1}$

Now we will construct the cut $\model{I}_n$. From Theorem \ref{prtot} it follows that $\PA\vdash \conv{\fast_{\omega_{n+1}}}$  and thus $\model{M}\models \conv{\fast_{\varepsilon_0}(n)}$. Let us denote by $u\in\model{M}$ the non-standard number such that $$\model{M}\models \fast_{\varepsilon_0}(n)=u.$$ Since $\conv{\fast_{\varepsilon_0}(n)}$ is a $\Sigma_1$-sentence, for every cut $\model{I}$ such that $\model{I}<u$, we have $\model{I}\models \dive{\fast_{\varepsilon_0}(n)}$.

Thus it is sufficient to take as $\model{I}_n$ any cut $\model{I}$ such that  $\model{I}<u$ and $$\model{I}\models \IDOE+\TI{\Pi_1}{\omega_{n+1}}.$$ Let us denote by $a\in\model{M}$ the nonstandard number such that $$\model{M}\models a=\fast_{\varepsilon_0}(n_0)=\fast_{\omega_{(n_0+1)}}(n_0).$$ We denote by $b$ the nonstandard number such that $$\model{M}\models b=\fast_{\omega_{n}\cdot a}(a).$$ Now we can apply Theorem \ref{TI_on_cut} with $c=a$ and obtain a cut $\model{I}$ such that $a<\model{I}<b$ and $\model{I}\models \IDOE+\TI{\Pi_1}{\omega_{n+1}}$.

 We claim that $\model{M}\models u>b=\fast_{\omega_n\cdot a}(a)$ and thus that $\model{I}_n<u$. Clearly, for any $k<n+1$ we will get a limit ordinal by applying operation $\alpha\longmapsto \alpha[n]$ exactly $k$ times to $\omega_{n+1}$.  Thus $n+1$ is the least $r$ such that $\omega_{n+1}\stepdown[r]{n}\beta+1$ for some $\beta$ . We denote by $\beta$ the ordinal such that $\omega_{n+1}\stepdown[n+1]{n}\beta+1$.  Since $n> n_0\ge 2$, the ordinal $\beta$ is greater than $\omega_n^2$.   Also, $\beta$ is greater than $\omega_{n_0+1}$.  Using induction we derive that $\omega_{n+1}\stepdown{n}\omega_n^2$ and $\omega_{n+1}\stepdown{n}\omega_{n_0+1}$ from Lemma \ref{stepdown_exp}. Since $\beta>\omega^2_n$ and $\beta>\omega_{n_0+1}$, we have $\beta\stepdown{n}\omega_n^2$ and $\beta\stepdown{n}\omega_{n_0+1}$. Since $\model{M}$ satisfies $\IDOE$, we can use the definition of fast-growing hierarchy to obtain $$\model{M}\models u=\fast_{\varepsilon_0}(n)=\fast_{\beta}^{n+1}(n).$$ From Lemma  \ref{fasthier_IDOE_facts} it follows that $\model{M}\models \omega_{n-1}\stepdown{a} 0$. By Lemma \ref{stepdown_aux2} we have $\model{M}\models \omega_{n-1}^2\stepdown{a} \omega_{n-1}+1$. Therefore by Lemma \ref{stepdown_exp} we have  $\model{M}\models \omega_{n}^2\stepdown{a} \omega^{\omega_{n-1}+1}$. Thus $\model{M}\models  \omega_n^2\stepdown{a}\omega_n\cdot a$. From Lemma \ref{fasthier_IDOE_facts} and Lemma \ref{stepdown_aux4} it follows that $$\begin{aligned} \model{M}\models \fast_{\beta}^{n+1}&(n) >\fast_{\beta}^2(n)\ge \fast_{\omega_{n}^2}(\fast_{\omega_{n_0+1}}(n))\ge \\ & \fast_{\omega_{n}^2}(\fast_{\omega_{n_0+1}}(n_0))=\fast_{\omega_{n}^2}(a)\ge\fast_{\omega_{n}\cdot a}(a).\end{aligned}$$ Hence we have proved our claim.

Thus for every $n$ and $\alpha<\omega_{n+1}$ we have a model of $\IDOE+\Diamond\!_{n}\DiamondSl\!_1^{\alpha}\varphi$. Thus for every $n$ and $\alpha<\varepsilon_0$ we have a model of $\IDOE+\Diamond\!_{n}\DiamondSl\!_1^{\alpha}\varphi$, since for every $\psi$ and $n_1<n_2$ we have $\IDOE\vdash \Diamond\!_{n_2}\psi\to \Diamond\!_{n_1}\psi$. Note that every $\Pi_1$-sentence that holds in some model of $\IDOE$ is true. Thus for every $\alpha<\varepsilon_0$ and number $n$   we have  $\Diamond\!_{n}\DiamondSl\!_1^{\alpha}\varphi$. Therefore, for every $\alpha<\varepsilon_0$ we have $\Diamond\DiamondSl\!_1^{\alpha}\varphi$. And therefore $\DiamondSl\!_1^{\varepsilon_0}\varphi$.

Now assume  that $\DiamondI[1]\top$.  Suppose we have been given an arithmetical sentence $\varphi$ such that $\Diamond\varphi$. By Löb's Theorem for $\DiamondI[1]$ we have $\DiamondI[1]\lnot \DiamondI[1]\top$.  Thus because $\Diamond\varphi$ is $\Pi_1$, we have $\DiamondI[1](\lnot \DiamondI[1]\top\land \Diamond\varphi)$. Therefore we have $\Diamond(\lnot \DiamondI[1]\top\land \Diamond\varphi)$. Note that formalization of the first part of the proof gives us $$\PA\vdash\lnot \DiamondI[1]\top \to \forall \psi(\Diamond\psi\to \forall \alpha<\varepsilon_0 \Diamond \DiamondSl\!_1^\alpha\psi).$$
Using that we conclude that $\Diamond(\forall \alpha<\varepsilon_0 \Diamond \DiamondSl\!_1^\alpha\varphi)$. Next we see that  $\forall \alpha<\varepsilon_0 \Diamond \DiamondSl\!_1^\alpha\varphi$ is a $\Pi_1$-sentence. Therefore we have $\forall \alpha<\varepsilon_0 \DiamondSl\!_1 \DiamondSl\!_1^\alpha\varphi$, i.e. $\DiamondSl\!_1^{\varepsilon_0}\varphi$.\end{proof}

Using Lemma \ref{epsilon_0_root_lower_bound} and Lemma \ref{epsilon_0_upper_bound},  we obtain the following  theorem:

\begin{theorem}[$\PA$]\label{Ite_Corollary}
 Suppose $\varphi$ is a sentence  and $\alpha<\varepsilon_0$ is a non-zero ordinal. Then the following sentences are equivalent:
\begin{enumerate}
\item $\Box\varphi$,
\item $\BoxSl_1^{\varepsilon_0}\varphi$,
\item $\Box\BoxSl_1^{\alpha}\varphi$.
\end{enumerate}
\end{theorem}

Using  Lemma \ref{iterations_monotonicity_for_predicates} and Lemma \ref{epsilon_0_upper_bound}  we can generalize the previous theorem to a wider spectrum of provability predicates:
\begin{theorem}  \label{Epsilon_note_root_theorem}Suppose  $\thr{S}_n$ is a recursive sequence of finitely axiomatizable theories such that $\PA$ proves that $\PA=\bigcup \limits_{n\in \omega} \thr{S}_n$  and $\thr{S}_n\subset \ISi{n + 1}$. Let $\triangle$ denote the provability predicate $\Box_{\thr{S}_x,\fast_{\varepsilon_0}}^{\varepsilon_0}$. Then $\PA$ proves that for every arithmetical sentence $\varphi$,  and non-zero ordinal $\alpha<\varepsilon_0$  the following sentences are equivalent:
\begin{enumerate}
\item $\Box\varphi$,
\item $\triangle^{\varepsilon_0}\varphi$,
\item $\Box\triangle^{\alpha}\varphi$.
\end{enumerate}
\end{theorem}


\section{Square root of $\PA$-provability}\label{Sqrt_Section}
Recall that for a recursive function $\recfun{f}$ we denote by $\Box_{\recfun{f}}$ the slowdown of standard $\PA$-provability predicate by the function $\recfun{f}$:
$$\Box_{\recfun{f}}\varphi \mathrel{:{=}}\Box_0\varphi \mathop{\lor} \exists y\le x (\Box_y \varphi\mathop{\land} \conv{\recfun{f}(x)}).$$

We are going to prove the following theorem.
\begin{theorem}\label{Square_root_theorem}There exists a fast-growing recursive function $\recfun{r}$ such that
$$\PA\vdash\forall \varphi(\Box_{\recfun{r}}\Box_{\recfun{r}}\varphi\mathrel{\leftrightarrow}\Box \varphi).$$
\end{theorem}

First we define the required function $\recfun{r}$ and auxiliary function $\recfun{l}$. The function $\recfun{r}$ will be a function that grows faster than any $\fast_{\omega_n}$, but considerably slower than $\fast_{\varepsilon_0}$. 
\begin{definition}
We define function $\recfun{l}$ by recursion. We give $\recfun{l}(n)$ under the assumptions that for all $m<n$ the values $\recfun{l}(m)$ are already defined:
$$\recfun{l}(n)=\max (\{0\}\cup \{m\mid 0<m<n\mbox{ and } \fast_{\omega_{\recfun{l}(m)}}(m)\le n\}).$$  
We define the function $\recfun{r}$ as following: $$\recfun{r}(n)=\fast_{\omega_{\recfun{l}(n)}}(n).$$
\end{definition}

\begin{lemma}[$\IDOE$] \label{l_g_properties} The following holds:
\begin{enumerate}
\item \label{l_g_properties_0}$\recfun{l}$ is total;
\item \label{l_g_properties_1}there exists $n $ such that $\recfun{l}(n )=1$;
\item \label{l_g_properties_2}function $\recfun{l}$ is monotone-nondecreasing;
\item \label{l_g_properties_3}for every $n $, we have $\conv{\recfun{r}(\recfun{l}(n ))}$;
\item \label{l_g_properties_4}for every $n $, we have $\recfun{r}(\recfun{l}(n ))\le n $ if $\recfun{l}(n )\ge 1$;
\item \label{l_g_properties_5}for every $n $ such that $\conv{\recfun{r}(n )}$, we have $\recfun{l}(\recfun{r}(n ))=n $.
\end{enumerate}
\end{lemma}
\begin{proof} Straightforward using definition of $\recfun{l}$, $\recfun{r}$  and Lemma \ref{fasthier_IDOE_facts}.
\end{proof}

\begin{lemma}[$\ISi{1}$] \label{g_monotonicity} For every $n $, if $\conv{\recfun{r}(n )}$ then for all $m <n $ we have $\conv{\recfun{r}(m )}$ and $\recfun{r}(m )<\recfun{r}(n )$.
\end{lemma}
\begin{proof}Follows from Lemma \ref{fasthier_IDOE_facts} Item 1 and easily provable  fact that if $\conv{\fast_{\alpha}(n)}$ and $n>m$ then $\fast_{\alpha}(n)>\fast_{\alpha}(m)$.\end{proof}

\begin{lemma}[$\ISi{1}$] \label{sqrt_PA_box_imp}For every arithmetical sentence $\varphi$, if we have $\Box \varphi$ then we have $ \Box_{\recfun{r}}\Box_{\recfun{r}}\varphi$. 
\end{lemma}
\begin{proof} Suppose we have $\Box \varphi$. Then there is a number $n $ such that $\Box_n\varphi$. Let  $m =\recfun{l}(n)$. By Lemma \ref{Box_n_prov_conv} we have $\Box_m\conv{\fast_{\omega_{m}}(n)}$. Hence directly from the definition of $\recfun{r}$ we have $\Box_m\conv{\recfun{r}(n)}$. Thus $\ISi{m}$ proves that $\ISi{n} \subseteq \PAsl{\recfun{r}}$, i.e. $\Box_m \forall \psi ( \Box_n \psi\to \Box_{\recfun{r}}\psi)$. By $\Sigma_1$-completeness of $\ISi{m} $ we have $\Box_m \Box_n\varphi$. Combining conclusions of the last two sentences, we obtain  $\Box_m\Box_{\recfun{r}}\varphi$. By Lemma \ref{l_g_properties} Item \ref{l_g_properties_3} we have $\conv{\recfun{r}(m )}$ and therefore $\ISi{m }\subseteq \PAsl{\recfun{r}}$. Thus $\Box_{\recfun{r}}\Box_{\recfun{r}}\varphi$.
\end{proof}

\begin{theorem}$\mbox{\cite[Theorem~3.7]{FriRatWei13}}$ \label{ISi_on_segment} Suppose $\model{M}$ is a non-standard model of the theory  $\IDOE$,  $n\ge 1$ is a standard integer, and $a,b,e\in\model{M}$ are non-standard integers such that $\model{M}\models\fast_{\omega_{n-1}^{e}}(a)=b$.  Then there is a cut $\model{J}$ of $\model{M}$ such that $a<\model{J}<b$ and $\model{J}$ is a model of $\ISi{n}$.
\end{theorem}

\begin{remark} Theorem \ref{ISi_on_segment} can be formalized in $\PA$ in the same fashion as Theorem \ref{TI_on_cut} in Remark \ref{TI_on_cut_form}. We assume that the model $\model{M}$ is given by some fixed tuple of arithmetical formulas, possibly with additional parameters. Then there are arithmetical formulas $\varphi(x,n,a,b,e)$ and $\psi(y,n,a,b,e)$  such that $\PA$ proves that if $n,a,b,e$ satisfy the conditions of Theorem \ref{ISi_on_segment} then the set of all $x$ from $\model{M}$ for which $\varphi(x,n,a,b,e)$ is a cut $\model{I}$, the formula $\psi(y,n,a,b,e)$ gives a $\Pi_{n}$ partial satisfaction relation for $\model{I}$, and the cut $\model{I}$ satisfies the conclusion of Theorem \ref{ISi_on_segment}.
\end{remark}

\begin{lemma}[$\PA$] \label{sqrt_PA_con_imp} For every arithmetical sentence $\varphi$, if  $\Diamond\varphi$ then $\Diamond\!_{\recfun{r}}\Diamond\!_{\recfun{r}}\varphi$.
\end{lemma}
\begin{proof} Let us consider a model $\model{M}$ of $\PA+\varphi$. If $\model{M}\models \Diamond\!_{\recfun{r}}\varphi$ then we are already done. So we assume that $\model{M}\not \models \Diamond\!_{\recfun{r}}\varphi$. Thus there exist  $u \in\model{M}$ such that $\model{M}\models \Diamond\!_u\varphi\land \lnot \Diamond\!_{u+1}\varphi\land \conv{\recfun{r}(u +1)}$.  From Lemma \ref{g_monotonicity} it follows that $\model{M}\models \conv{\recfun{r}(u )}$ and $\recfun{r}(u )<\recfun{r}(u +1)$.  Note that $u $ is a non-standard number because $\PA$ is an essentially reflexive theory and thus for every standard $m $ we have $\model{M}\models \Diamond\!_{m}\varphi$.

We are going to prove that for every $n $ such that $\conv{\recfun{r}(n )}$ we have a cut $\model{J}_n$ of $\model{M}$ such that $\recfun{r}(u )<\model{J}_n<\recfun{r}(u +1)$,  the cut $\model{J}_n$ have $\Pi_n$ satisfaction relation, and $\model{J}_n\models \ISi{n }$. If we prove our claim then for every  cut $\model{J}_n$, we will have $\PAsl{\recfun{r}}=\ISi{u }$ within it and hence  $\model{J}_n\models \Diamond_{\recfun{r}}\varphi$.  Since every cut $\model{J}_n$ that we will construct  will have $\Pi_n$ partial satisfaction relation and will be a model of $\ISi{n}$, we can conclude that $\Diamond_n\Diamond_{\recfun{r}}\varphi$, for every $n$ such that $\conv{\recfun{r}(n)}$. Hence we will have $\Diamond_{\recfun{r}}\Diamond_{\recfun{r}}\varphi$, contradiction.

Now let us prove the claim.  We take $a,b,e\in\model{M}$: 
$$\model{M}\models a= \recfun{r}(u )\mbox{,$\,\,\,\,$} \model{M}\models b= \fast_{\omega^{u }_{n -1}}(a)\mbox{, and$\,\,\,\,$} \model{M}\models e=u $$ and apply Theorem \ref{ISi_on_segment} to construct the required cut. Now we need to check that we indeed have $\recfun{r}(u )<\model{J}_n<\recfun{r}(u +1)$. For this, it is enough to show that $$\model{M}\models  \fast_{\omega^{u }_{n -1}}(\recfun{r}(u))\le \recfun{r}(u+1).$$  Since $u$ is a non-standard number, we have $\model{M}\models\recfun{r}(n )<u $.  From Lemma \ref{l_g_properties} item \ref{l_g_properties_2}. it follows that $\model{M}\models\recfun{l}(u )\ge \recfun{l}(\recfun{r}(n ))$.  From Lemma \ref{l_g_properties} item \ref{l_g_properties_5}.   it follows that $\recfun{l}(\recfun{r}(n ))=n $.  Thus $\model{M}\models\recfun{l}(u +1)\ge \recfun{l}(u )\ge n $.  We see that $$\model{M}\models \recfun{r}(u )=\fast_{\omega_{\recfun{l}(u )}}(u )=\fast_{\omega^{u +1}_{\recfun{l}(u )-1}}(u )$$ and  $$\model{M}\models \recfun{r}(u+1 )=\fast_{\omega_{\recfun{l}(u +1)}}(u +1)=\fast_{\omega^{u +2}_{\recfun{l}(u +1)-1}}(u +1).$$ 

From Lemma \ref{fasthier_IDOE_facts} it follows that $\model{M}\models \omega^{u+2}_{\recfun{l}(u+1)-1}\stepdown{u+1} 0$. Now we use Lemma \ref{stepdown_aux3}  and Lemma \ref{stepdown_aux2} and deduce that $$\model{M}\models \omega^{u +2}_{\recfun{l}(u +1)-1}\stepdown{u+1} \omega^{u +2}_{\recfun{l}(u)-1} \stepdown{u+1} \omega^{u +1}_{\recfun{l}(u)-1}+1.$$ Therefore from  Lemma \ref{fasthier_IDOE_facts} it follows that $$\model{M}\models \fast_{\omega^{u +2}_{\recfun{l}(u +1)-1}}(u +1)\ge \fast_{\omega^{u +1}_{\recfun{l}(u)-1}+1}(u+1)\ge \fast_{\omega^{u +1}_{\recfun{l}(u)-1}+1}(u).$$ Now using the fact that $\model{M}\models\conv{\fast_{\omega^{u +1}_{\recfun{l}(u)-1}+1}(u)}$ and   Lemma \ref{fasthier_IDOE_facts} we conclude that $$\model{M}\models \omega^{u +1}_{\recfun{l}(u)-1}+1\stepdown{u} 0.$$ Next using Corollary  \ref{stepdown_aux3} and Lemma \ref{stepdown_aux3} we deduce that $$\model{M} \models \omega^{u +1}_{\recfun{l}(u) -1}\stepdown{u}\omega^{u }_{\recfun{l}(u) -1}\stepdown{u}\omega^{u }_{n -1}.$$ From Lemma \ref{fasthier_IDOE_facts} and Lemma \ref{stepdown_aux4} it follows that that $$\begin{aligned}\model{M}\models \fast_{\omega^{u +1}_{\recfun{l}(u)-1}+1}(u)= \fast_{\omega^{u +1}_{\recfun{l}(u)-1}}^{u+1}(u)\ge  \fast_{\omega^{u +1}_{\recfun{l}(u)-1}}^{2}(u)\ge  \fast_{\omega^{u }_{n -1}}(\fast_{\omega^{u +1}_{\recfun{l}(u )-1}}(u )).\end{aligned}$$  Therefore $$\model{M}\models  \fast_{\omega^{u }_{n -1}}(\recfun{r}(u))=\fast_{\omega^{u }_{n -1}}(\fast_{\omega^{u +1}_{\recfun{l}(u )-1}}(u ))\le\fast_{\omega^{u +2}_{\recfun{l}(u +1)-1}}(u +1)=\recfun{r}(u+1).$$ Thus we have obtained the required cut $\model{J}_n$.\end{proof}

From Lemma \ref{sqrt_PA_box_imp} and Lemma \ref{sqrt_PA_con_imp} it follows that Theorem \ref{Square_root_theorem} holds.


\section{Bimodal provability logics}\label{ProvLogic_Section}
We determine the joint provability logic of a wide class of pairs of provability predicates, including ordinary provability together with any of the slow provability predicates $\BoxSl_{z}$. 
This result was obtained in cooperation with Volodya Shavrukov. 
The relevant bimodal system is $\GLT$, or Lindstr\"om logic, 
first studied by 
Lindstr\"om (\cite{Lin06}) due to its relation to Parikh's rule. The latter allows one to infer $\varphi$ from $\fofl{Pr}_{\PA}(\varphi)$. 
Since Parikh's rule is admissible in $\PA$, adding it to $\PA$ does not yield new theorems. As shown in \cite{Par71}, it does yield speed-up, meaning that some theorems have much shorter proofs when Parikh's rule is allowed. The equivalence of 
Parikh provability and ordinary provability is however not verifiable in $\PA$.

In Section \ref{9squareroot} we establish the joint provability logic of ordinary provability together with square root provability.

\subsection{The system $\GLT$}
We work with the languages $\mathcal{L}_{\Box}$ and $\mathcal{L}_{\Box\triangle}$ of propositional unimodal and bimodal logic respectively. The symbols 
$\Box$ and $\triangle$ are thus used for the modalities, not as abbreviations for arithmetical formulas as until now. As before, $\Diamond A$ and $\triangleDual A$ are written as shorthand for $\lnot \Box\lnot A$ and 
$\lnot \triangle\lnot A$ respectively. 

\begin{definition}
The axiom schemata of $\GL$ include all propositional tautologies in the language $\mathcal{L}_{\Box}$, and furthermore:
\begin{flalign*}
\mathsf{(K)}\;\; & \Box(A\to B) \to (\Box A\to \Box B)& \\
\mathsf{(L)}\;\;& \Box(\Box A\to A)\to \Box A 
\end{flalign*}
The inference rules of $\GL$ are modus ponens and necessitation: 
if $\vdash_{\GL} A$ then $\vdash_{\GL} \Box A .$
\end{definition}
\begin{lemma}
$\vdash_{\GL}\Box A\to \Box \Box A$
\end{lemma}
\begin{proof}
See \cite[Theorem 1.18]{Boo93}. 
\end{proof}

A relation $\prec$ on a set $W$ is said to be \emph{conversely well-founded} if for every $S\subseteq W$ with $S\neq \emptyset$, there is some $a\in W$ such that $a\nprec b$ for all $b\in S$. A conversely well-founded relation is, in particular, irreflexive. 
We write $a \preceq b$ if either $a\prec b$ or $a=b$. 

\begin{definition}
A \emph{$\GL$-frame} $\krmod{F}$ is a tuple $\langle W, \prec \rangle$, where  $\prec$ is a transitive conversely well-founded relation on $W$. $\krmod{F}$ is said to be \emph{tree-like} if there is a root $r\in W$ such that for all $a\in W$, $r\preceq a$, and
furthermore each $a\neq r$ has a 
unique immediate $\prec$-predecessor $b$.
\end{definition}
\begin{definition}
A \emph{$\GL$-model} is a triple $\langle W, \prec, \Vdash \rangle$, where $\langle W, \prec \rangle$ is a $\GL$-frame, and 
 $\Vdash$ a valuation assigning to every propositional letter a subset of $W$. $\Vdash$ is extended to all formulas of $\mathcal{L}_{\Box}$  by 
requiring that it commutes with the propositional connectives, and interpreting $\prec$ as the accessibility relation for $\Box$:  
\begin{equation*}
\krmod{M}, a\Vdash \Box A  \mbox{ if for all } b \mbox{ with } a\prec b,
b \Vdash A. \end{equation*}
\end{definition}
Given $\krmod{M}= \langle W, \prec, \Vdash \rangle$, we write 
$\krmod{M}\Vdash A$ if $\krmod{M},a\Vdash A$ for every $a\in W$. We write 
$\krmod{F}\Vdash A$ if  $\krmod{M}\Vdash A$ for any model $\krmod{M}$ whose underlying frame is $\krmod{F}$.

\begin{theorem}\label{glmodcomp}
$\vdash_{\GL} A$ iff for every tree-like $\GL$-frame $\krmod{F}$,  $\krmod{F}\Vdash A$. \end{theorem}
\begin{proof}
See for example Chapter 5 of \cite{Boo93}. 
\end{proof}

\begin{definition} \label{glt}
The axiom schemata of 
$\GLT$ include all propositional tautologies in the language $\mathcal{L}_{\Box\bigtriangleup}$, 
the rules and axiom schemata of $\GL$ for $\triangle$, as well as:
\begin{flalign*}
\mathsf{(K^{\Box})}\;\; & \Box(A\to B) \to (\Box A\to \Box B) \\
\mathsf{(T1)}\;\; & \triangle A\to \Box A \\
\mathsf{(T2)}\;\; & \Box A\to \triangle\Box A \\
\mathsf{(T3)}\;\; & \Box A\to \Box\triangle A \\
\mathsf{(T4)}\;\; & \Box \triangle A\to \Box A &
\end{flalign*}
\end{definition}

It is not difficult to show that $\vdash_{\GLT} \Box(\Box A\to A)\to \Box A$.  
Thus $\GLT$ contains
$\GL$ for both $\triangle$ and $\Box$.

Given a tree-like $\GL$-frame $\langle W, \prec \rangle$ and $a, b\in W$, we write $a\prec_{\infty} b$ if the set $\{c\mid a\prec c \prec b\}$ is infinite. It is shown in \cite{Lin06} that $\GLT$ is sound and complete with respect to the class all of tree-like $\GL$-frames, with $\prec$ and $\prec_{\infty}$ as the accessibility relations for $\triangle$ and $\Box$ respectively.
We consider a slightly different semantics for $\GLT$ that --- while arguably less neat than the one just described --- has the advantage of allowing us to work with finite models. Indeed, it was introduced by  
Lindstr\"om in order to obtain decidability of $\GLT$.

\begin{definition}\label{gltmodel}
For $A\in \mathcal{L_{\Box\triangle}}$, an
\textit{$A$-sound model} is a quadruple $\langle W, \prec, \prec_{R}, \Vdash \rangle$, where
\begin{enumerate}
\item $\langle W, \prec \rangle$ is a finite tree-like $\GL$-frame
\item $a \prec_{R} b \,\Rightarrow\, a \prec b$
\item $a \prec b \prec_{R} c\, \Rightarrow \, a \prec_{R} c$
\item $a \prec_{R} b \prec c \,\Rightarrow \, a \prec_{R} c$
\item if $a\prec_{R}b$, there is some $c$ with $a\prec_{R}c \preceq b$, and such that  \label{refnode}
$c\Vdash \triangle B\to B$ for every subformula $B$ of $A$. Such a node $c$ shall be referred to as 
 \emph{reflexive}. 
\end{enumerate}
Finally, $\Vdash$ is a valuation satisfying the usual clauses, with $\prec$ and $\prec_R$ as the accessibility relations for $\triangle$ and $\Box$ respectively. 
\end{definition}

\begin{lemma}{\cite[Lemma 9]{Lin06}}\label{modalcomp}
Let $n$ be the cardinality of the set of subformulas of $A$.
$\vdash_\GLT A$ iff $\krmod{M}\Vdash A$ for every $A$-sound model of cardinality $\leq n^{n^2+1}$. 
\end{lemma}

\subsection{Arithmetical interpretations of modal logic}\label{ArInt}
In order to formulate the connection between $\GLT$ and our arithmetical provability predicates, we use the notion of an 
 \emph{arithmetical realization}.

\begin{definition}\label{realization}
Let $\fofl{Pr}_0$ and $\fofl{Pr}_1$ be provability predicates in the language $\mathcal{L}$ of arithmetic. An arithmetical realization $*$ is an assignment of $\mathcal{L}$-sentences to all modal formulas. The values of $*$ at propositional letters of the modal language can be arbitrary.
It is required that $*$ commutes with 
the propositional connectives, and furthermore
$(\Box A)^*:= \fofl{Pr}_0(\gnmb{A^*})$ and
$(\triangle A)^* = \fofl{Pr}_1(\gnmb{A^*})$. 
\end{definition}
We note that the values of an arithmetical realization $*$ are determined by its values at the propositional letters of $\mathcal{L}_{\Box\triangle}$.
\begin{definition}
Let $\fofl{Pr}_0$ and $\fofl{Pr}_1$ be provability predicates for a theory $\thr{T}$ containing $\IDOE$.
We say that $\GL$ ($\GLT$) is \emph{sound} for $\fofl{Pr}_0$ (and $\fofl{Pr}_1$) if:
$$\vdash_{\GL(\GLT)} A  \;\;\; \Rightarrow   \;\;\;\thr{T} \vdash A^* \mbox{ for all arithmetical realizations } *.$$
We say that $\GL$ ($\GLT$) is \emph{complete} for $\fofl{Pr}_0$ (and $\fofl{Pr}_1$) if:
$$ \thr{T}\vdash A^* \mbox{ for all arithmetical realizations } * \;\;\;\Rightarrow \;\;\; \vdash_{\GL(\GLT)} A.$$
If $\GL$ ($\GLT$) is both sound and complete with respect to $\fofl{Pr}_0$ (and $\fofl{Pr}_1$), we say that $\GL$ ($\GLT$) is the \emph{provability logic} of $\fofl{Pr}_0$ (and $\fofl{Pr}_1$). This means that $\GL$ ($\GLT$)  contains exactly those principles of provability in $\thr{T}$ --- as given by $\fofl{Pr}_0$ and $\fofl{Pr}_1$ respectively --- 
that can be verified in $\IDOE$.
\end{definition}

\begin{theorem}\label{glarcomp}
Let $\thr{T}\supseteq\IDOE $ be recursively axiomatizable and $\Sigma_1$-sound. Then $\GL$ is the provability logic of $\fofl{Pr}_{\thr{T}}$.  
\end{theorem}
\begin{proof}
The case where $\thr{T}=\PA$ was proven by Solovay in \cite{Sol76}. Extension to $\Sigma_1$-sound recursively axiomatizable theories containing $\IDOE$ is due to de Jongh, Jumelet, and Montagna in \cite{Dej91}. 
\end{proof}


Throughout the rest of this section, we use the symbols $\Box$ and $\triangle$ for 
both provability predicates and modalities. Since lower-case Greek letters range over arithmetical and upper-case Latin letters over modal sentences, the intended meaning will always be clear from the context. 
With this notation, we also imply that we are interested in arithmetical realizations interpreting the \emph{modalities} $\triangle$ and $\Box$ as the corresponding provability predicates. Bearing this in mind we shall, from now on, mostly skip referring to the explicit formalization of Definition \ref{realization}.

\subsection{Arithmetical soundness and completeness} 
We prove the following arithmetical completeness theorem.

\begin{theorem}\label{GLTcompl}
Let $\Box$ and $\triangle$ be $\Sigma_1$-provability predicates numerating the 
same sound theory $\thr{T}$ extending $\IDOE$, i.e.\ for all $\varphi$,
\begin{enumerate}
\item $\thr{T}\vdash \varphi$ if and only if $\IDOE\vdash \Box\varphi$
\item $\model{N}\vDash \Box \varphi\leftrightarrow \triangle \varphi$, where $\model{N}$ is the standard model
\end{enumerate}
Assume furthermore that $\GLT$ is sound for $\Box$ and $\triangle$. Then $\GLT$ is also complete for $\Box$ and $\triangle$.
\end{theorem}


\begin{remark}
With the above interpretation of $\Box$, we depart from the convention that $\Box$ stands for the ordinary provability predicate of $\PA$. 
\end{remark}

\begin{lemma}
Let $\{\thr{S}_n\}_{n\in\omega}$ be a recursive sequence of finitely axiomatizable theories such that 
$\PA$ proves $\PA=\bigcup_{n\in\omega}\thr{S}_n$. Write $\triangle$ for $\Box_{\thr{S}_x, \fast_{\varepsilon_0}}$, and $\Box$ for the usual provability predicate of $\PA$. 
Then $\PA$ verifies the following, for all $\varphi$: 
\begin{enumerate}
\item $\triangle\varphi \to \Box \varphi$ \label{eins}
\item $\Box \varphi \to \triangle \Box\varphi $ \label{zwei}
\item $\Box \varphi \to \Box \triangle \varphi$ \label{drei}
\item $\Box \triangle \varphi \to \Box\varphi$ \label{vier}
\end{enumerate}
\end{lemma}
\begin{proof}
Item \ref{eins} is clear from the definition, item \ref{zwei} follows by provable $\Sigma_1$-completeness of $\triangle$, \ref{drei} is Corollary \ref{slowcon_inv} in Section \ref{Converting_Section}, and \ref{vier} is Theorem \ref{slowcon} in Section \ref{slowprovability}. 
\end{proof}
By Remark \ref{shifting} in Section \ref{slowprovability}, it follows that the requirements of Theorem \ref{GLTcompl} are satisfied when taking for $\Box$ the usual provability predicate of $\PA$, and for $\triangle$ any of the slow provability predicates $\BoxSl_{z}$. We note that the requirements of Theorem \ref{GLTcompl} are also satisfied when taking for $\triangle$ the usual provability predicate of $\PA$, and for $\Box$ the provability predicate for $\PA$ together with Parikh's rule\footnote{Lindstr\"om's proof of arithmetical completeness of $\GLT$ with respect to ordinary provability $\Box$ and Parikh provability $\Box_{\mathsf{p}}$ made essential use of the fact that 
$\PA \vdash \forall \varphi \left( \Box_{\mathsf{p}} \varphi\leftrightarrow \Box^{\omega} \varphi \right)$.
While, as follows from the results of Section \ref{Converting_Section}, the same relation holds between $\Box$ and $\BoxSl_2$, it fails for $\Box$ and $\BoxSl_1$, where we only have  
$\PA \vdash \forall \varphi \left(\Box \varphi \leftrightarrow \BoxSl_1^{\varepsilon_0} \varphi\right)$.
Our proof method does not rely on $\triangle$ being a certain ordinal iteration of $\Box$, and is therefore applicable to a wider class of predicates.}.


We assume all $\Sigma_1$-sentences to be of the form $\exists y\, \psi$, with 
$\psi$ a $\Delta_0$-formula. If $\varphi$ is a $\Sigma_1$-sentence, we write $n: \varphi$ to mean that $n$ is a witness for $\varphi$, i.e.\ if 
$\psi(\overline{n})$ holds. We also assume: 
\begin{enumerate}
\item if $n: \Box\varphi$, then for any $\psi\neq \varphi$ it is not the case that $n:\Box\psi$ \label{teine}
\item if there is some $n$ with $n:\Box \varphi$, then there are arbitrarily large $n'$ with $n':\Box \varphi$ \label{kolmas}
\end{enumerate}
The above requirements hold for any reasonable arithmetization of syntax in arithmetic. For every $\Sigma_1$-formula $\varphi$, there exists a formula $\exists y\, \psi$ satisfying 
\ref{teine} and \ref{kolmas} that is $\IDOE$-provably equivalent to 
$\varphi$; thus the above assumption does not restrict us in any way.

Let $\thr{T}$, $\Box$ and $\triangle$ be as in the statement of Theorem \ref{GLTcompl}. 
Our proof of arithmetical completeness proceeds, as usual, 
by constructing a
suitable Solovay function moving along the accessibility relations of a $\GLT$-model. 
For the remainder of this section, fix some $A$-sound model $\krmod{M}=\langle W, \prec, \prec_{R}, \Vdash \rangle$. We assume that $\krmod{M}$ has a root, i.e.\  that there is a
node $0\in W$ such that $0\prec a$ for every $0\neq a\in W$. We also assume as given $\mathcal{L}$-formulas representing $\krmod{M}$ in $\IDOE$ in a natural way.
\begin{definition}{\textup($\IDOE$\textup)}\label{solovay} 
Define the primitive recursive function $h:\omega \to W$:
\begin{align*}
h(0)&=0\\
h(n+1) &= \begin{cases} b & \mbox{ if }  h(n)\prec_{R}b, \,b \mbox{ is reflexive}, \mbox{ and } n:\Box L\neq b, \mbox{ else:}\\
c & \mbox{ if } h(n)\prec c\;  \mbox{ and } n: \triangle L\neq c\\
h(n) & \mbox{ otherwise } \end{cases} 
\end{align*}
For $a\in W$, $L= a$ is the formula expressing that $a$ is the limit of $h$. 
As usual, the apparent circularity in the definition of $h$ is dealt with by using the Diagonal Lemma (see \cite{Sol76} or \cite{Boo93}). 
\end{definition}

The function $h$ starts at the root $0$, and moves along the relations $\prec$ and $\prec_{R}$. It makes an  $\prec_{R}$-step to some node $b$ only if there is a $\Box$-proof that it will not stay there, and it makes an $\prec$-step to some node $c$ only if there is a $\triangle$-proof that it will not stay there.
We refer to the elements of the domain of $h$ as \emph{stages}, saying for example that $h$ moved to $a$ at stage $n$ if $h(n-1)\neq h(n)=a$.

\begin{lemma}\label{limitlemma}{\textup($\IDOE$\textup)} $h$ has a unique limit. 
\end{lemma}
\begin{proof}
Since $a\prec_{R}b$ implies $a\prec b$, it is clear from the definition that $h$ moves along the $\prec$-relation in $\mathcal{M}$. Since $\prec$ is transitive, we thus have that
$h(x)\preceq h(y)$ whenever $x<y$. If $h$ were to keep moving forever, there would be an infinite ascending $\prec$-chain in $\krmod{M}$, contradicting the assumption that $\langle W, \prec\rangle$ is a $\GL$-frame. 
\end{proof} 
Let $\varphi_{h}(x,y)$ be an $\mathcal{L}$-formula representing $h$ in $\IDOE$ in a natural way. According to Lemma \ref{limitlemma}, \begin{equation}\label{defexp}
\IDOE\vdash \exists ! y\,\exists x_0 \forall x\geq x_0\, \varphi_{h}(x,y).
\end{equation}
In the remainder of this section, we work in a definitional expansion of $\IDOE$ that contains a term $L$ denoting the limit of $h$. We note that by (\ref{defexp}) such a definitional expansion is a conservative extension of $\IDOE$.
\begin{lemma}{\textup($\IDOE$\textup)}\label{move} If $h(n)=a$, then $a\preceq L$. 
\end{lemma}

\begin{proof}
If $h(n)=a$, then either $h$ stays at $a$, i.e.\ we have $L=a$, or it moves away from $h$, in which case all 
its further values are $\prec$-successors of $a$, and so $a\prec L$
(as in the proof of Lemma \ref{limitlemma}).
\end{proof} 

\begin{lemma}{\textup($\IDOE$\textup)}\label{nabla}
If $ L=a$ and $a\prec b$, then $\triangleDual L=b$.
\end{lemma}

\begin{proof}  Assume $L=a$ and $a\prec b$, and let $t$ be such that $\forall x\geq t\, h(x)=a$.
Suppose for a contradiction that $\triangle L\neq b$ holds. Then there also exists some 
some $n\geq t$ with $n: \triangle L\neq b$. But this means that $h$ moves away from $a$ to $b$ at stage $n$, contradicting our assumption that $L=a$. 
\end{proof}

\begin{lemma}{\textup($\IDOE$\textup)}\label{diamond}
If $L=a$ and $a\prec_{R}b$, then $\Diamond L=b$. 
\end{lemma}
\begin{proof}  
Assume $L=a$ and $a\prec_{R}b$. By properties of $\krmod{M}$, there exists some reflexive $c$
with $a\prec_{R} c$ and $c\preceq b$. Arguing as in the proof of Lemma \ref{nabla}, 
we have that $\Diamond L=c$. If $c=b$, we are done. If $c\prec b$, we have 
$L=c \to \triangleDual L=b$ by Lemma \ref{nabla}, whence also 
\begin{equation*}
\Diamond L=c \to \Diamond \triangleDual L=b
\end{equation*}
by modal reasoning, using the soundness of $\GL$ for $\Box$. An application of modus ponens yields 
$\Diamond \triangleDual L=b$, and therefore $\Diamond L=b$
by principle $\mathsf{T3}$ of $\GLT$. 
\end{proof}
\begin{lemma}{\textup($\IDOE$\textup)}\label{sigmaone}
If $L=a$, then $\triangle\, a\preceq  L$.
\end{lemma}
\begin{proof}
From $L=a$ it follows that  $h(n)=a$ for some $n$. 
By $\Sigma_{1}$-completeness of $\triangle$, 
\begin{equation*}
\triangle h(\dot{n})=a
\end{equation*}
Since the theory defined by $\triangle$ contains $\IDOE$,  Lemma  \ref{move} gives us
\begin{equation*}
\triangle \left( h(\dot{n})=a \to a\preceq L \right).
\end{equation*}
Combining the above yields $\triangle\, a\preceq L$ as required.
\end{proof}

\begin{lemma}{\textup($\IDOE$\textup)}\label{notrefl}
If $a\neq 0$ is not reflexive, then 
$L=a$ implies $ \triangle\, a \prec L$.
\end{lemma}

\begin{proof}
If the limit of $h$ is some irreflexive $a\neq 0$, then 
$h$ must have moved to $a$ due to some number witnessing $\triangle L\neq a$.
By lemma \ref{sigmaone}, we also have $\triangle\, a\preceq L$. Combining these, we get 
$\triangle\,  a\prec L$.
\end{proof}

\begin{lemma}{\textup($\PA$\textup)}\label{box}
If $a\neq 0$, then $ L=a$ implies $\Box\,  a\prec_{R}L$.
\end{lemma}
\begin{proof}
Assume $L=a$. By Lemma \ref{sigmaone}, we have $\triangle\, a\preceq L$, whence also $\Box\, a \preceq L$. Since $a\neq 0$ and $h$ moved  to $a$, it must be that $\Box L\neq a$.
It therefore suffices to show that if $b$ is such that $a\prec b$ but $a\nprec_{R}b$, then $\Box L\neq b$.
We assume that the claim has been established for all $b'$ with $a\prec b'\prec b$, i.e.\ that we have
 $\Box\, L\neq b'$ for all such $b'$. 

Since $L=a\neq 0$, $h$ must have moved to $a$ at some stage $t$.
Then either $\Box\, L\neq a$ or $\triangle\, L\neq a$ holds, but in either case $\Box\, L\neq a$. Argue in $\Box$: 

We have $h(t)=a$, $L\neq a$, and $L\neq b'$ for all $b'$ with $a\prec b'\prec b$. Thus $h$ has to eventually move away from any node in
$ \{a\}\cup \{b'\mid a\prec b'\prec b\}$. Consider the first stage where this happens, i.e.\ let $s>t$ be minimal 
with  $h(s) \notin \{a\}\cup \{b'\mid a\prec b'\prec b\}$, and let $c=h(s)$. Since $h(t)=a$ and $t<s$, we have $a\prec c$.
We consider two cases: 
\begin{enumerate}
\item $c=b$: the reason for moving could only have been that $s-1:\triangle\, L\neq b$. If the move had been
due to the first clause of Definition \ref{solovay}, we would  have $a\prec c\prec_{R}b$, and so $a\prec_{R}b$ by properties of $\krmod{M}$, contradicting our assumption that $a\nprec_{R}b$.

\item $c\neq b$: by the proof of Lemma \ref{sigmaone}, $h(s)=c$ implies 
$\triangle\, c\preceq L$. Our choice of $c$ excludes  $c\prec b$, and we have assumed $c\neq b$. Thus $\triangle L\neq b$ also in this case.
\end{enumerate}

Returning to the outside world, we have shown that $\Box\triangle\, L\neq b$. With principle $\mathsf{T4}$, we obtain $\Box\, L\neq b$ as desired.
\end{proof}

\begin{lemma}\label{truth}
Let $\model{N}$ be the standard model. Then 
\begin{enumerate}
\item $\model{N}\vDash L=0$
\item for any $a\in W$, the sentence $L = a$ is consistent with $\thr{T}$ \label{raherahe}
\end{enumerate}
\end{lemma}
\begin{proof}
By  Lemma \ref{limitlemma} (and soundness of $\IDOE$), $\mbox{$\model{N}\vDash L=a$}$ holds for some $a$. If $a\neq 0$, then either  $\model{N}\vDash \Box\, L\neq a$ or $\model{N}\vDash \triangle\, L\neq a$, depending on how $h$ moved to $a$. 
Since $\Box$ and $\triangle$ have the same extension in $\model{N}$, we have $\model{N}\vDash \Box\, L\neq a$ in both cases, and thus 
 $\thr{T}\vdash L\neq a$. By soundness of $\thr{T}$, we get 
 $\model{N}\vDash L\neq a$, contradicting our assumption. Thus it must be that
 $a=0$. 
For \ref{raherahe}, note that by Lemma \ref{nabla} (and soundness of $\IDOE$) $\model{N}\vDash L=0\to \triangleDual L=a$ holds for all $a\in W$. Since $\model{N}\vDash L=0$, it follows that 
$\model{N}\vDash \triangleDual L=a$, i.e.\ $\model{N}\vDash \lnot \triangle\, L\neq a$ for any $a\in W$. 
Suppose that $\thr{T} \vdash L\neq a$. Then $\model{N}\vDash \Box L\neq a$ and thus also
 $\model{N}\vDash \triangle L\neq a$, a contradiction. 
 Thus we conclude that $\thr{T} \nvdash L\neq a$. 
\end{proof}

\begin{lemma}\label{truthlemma}
Define the arithmetical realization $*$ by: 
\begin{equation}
p^*:=\bigvee_{\krmod{M},a\Vdash p} L=a .
\end{equation}
Then for every subformula $B$ of $A$, and every $a\in \krmod{M}$, $a\neq 0$
\begin{equation}
\krmod{M}, a\Vdash B \;\;\; \Rightarrow\;\;\;  \thr{T} \vdash L=a \to  B^{*}.
\end{equation}
\end{lemma}

\begin{proof}
This is proven in the standard manner by induction on the complexity of $B$, using Lemmas \ref{limitlemma}-\ref{box} (see e.g.\ \cite{Sol76} or \cite{Boo93}).  We treat the only slightly deviant case of $\triangle B$. 
Assume $a\Vdash \triangle B$, i.e.\ $b\Vdash B$ for all $b$ with $a\prec b$. 
By the induction assumption, we have 
$\thr{T}\vdash L=b \to B^*$
for all such $b$, and thus 
\begin{equation}\label{sausti}
\thr{T} \vdash a \prec L \to B^*.
\end{equation}
By modal reasoning, this implies 
\begin{equation}\label{kiili}
\thr{T}\vdash \triangle a \prec L \to \triangle B^*.
\end{equation}
Now argue in $\thr{T}$, assuming $L=a$.
If $a$ is irreflexive, then $L=a$ implies $\triangle a\prec L$ by Lemma \ref{notrefl}, and so $ \triangle B^*$ by (\ref{kiili}). If $a$ is reflexive, then we have $a\Vdash B$, whence the induction assumption additionally gives
 $L=a \to B^*$, and therefore $a\preceq L \to B^*$ and $\triangle a\preceq  L \to \triangle B^*$. With Lemma \ref{sigmaone}, $L=a$ implies $\triangle a\preceq L$, thus we obtain $\triangle B^*$ also in this case.
\end{proof}
We are now ready to combine the results of this section to prove Theorem \ref{GLTcompl}. 
\begin{proof}
We need to show that if $\nvdash_{\GLT} A$, there is some arithmetical realization $*$ with $\thr{T} \nvdash A^*$. 
Suppose that $\GLT\nvdash A$. By Lemma \ref{modalcomp}, let $\krmod{M}$ be an $A$-sound model with 
$\mathcal{M}, w\nVdash A$ for some $w\in \krmod{M}$.  We append a root $0$ to $\mathcal{M}$, and apply Definition \ref{solovay} to the resulting model.  By Lemma \ref{truthlemma}, $\thr{T} \vdash L=w\to \lnot A^*$. Therefore, since $\thr{T}\nvdash L\neq w$ by Lemma \ref{truth}, 
it must be that $\thr{T} \nvdash A^{*}$. 
\end{proof}

\subsection{Square root provability}\label{9squareroot}
We conclude with a characterization of the provability logic of square root and ordinary provability. In Section \ref{Sqrt_Section}, it was shown that there is a provability predicate $\Box_{\mathsf{r}}$ such that 
\begin{equation}\label{sinilind}
\PA\vdash\forall \varphi\,\left(\Box_{\mathsf{r}}\Box_{\mathsf{r}}\varphi\leftrightarrow  \Box \varphi\right),
\end{equation}
where $\Box$ denotes the usual provability predicate of $\PA$. 
By Theorem \ref{glarcomp}, $\Box_{\mathsf{r}}$ satisfies the Hilbert-Bernays-L\"ob derivability conditions (verifiably in $\IDOE$). As we will see, the  
equivalence in (\ref{sinilind}) is in fact sufficient for obtaining \emph{all} propositional schemata 
concerning $\Box$ and $\Box_{\mathsf{r}}$ that are provable in $\PA$. In the remainder of this section, we shall introduce 
the bimodal logic $\GLtwo$, and sketch the proof of its arithmetical completeness with respect to $\Box$ and $\Box_{\mathsf{r}}$.
\begin{definition}
The axiom schemata of $\GLtwo$ include all propositional tautologies in the language $\mathcal{L}_{\Box\triangle}$, 
the rules and axioms of $\GL$ for $\triangle$, as well as:
\begin{flalign*}
\mathsf{(2)}\;\;& \Box A \leftrightarrow \triangle\triangle A&
\end{flalign*}
\end{definition}

Given a tree-like $\GL$-frame $\langle W, \prec \rangle$ and $a, b\in W$, write $a\prec_{2} b$ if there is some $c$ with $a\prec c\prec b$. 

\begin{lemma}\label{gltwomodal}
$\GLtwo$ is sound and complete with respect to the class of tree-like $\GL$-frames, with $\prec$ and $\prec_{2}$ as the accessibility relations for $\triangle$ and $\Box$ respectively.
\end{lemma}
\begin{proof}
Easy exercise.
\end{proof}

It is proven in \cite[Section 7]{Vis90} that $\GLtwo$ is the joint provability logic of $\IDOE$-provability and cut-free provability in $\IDOE$.

\begin{theorem}
$\GLtwo$ is the provability logic of $\Box_{\recfun{r}}$ and $\Box$. 
\end{theorem}

\begin{proof}

Suppose that $\nvdash_{\GLtwo} A$. By Lemma  \ref{gltwomodal}, there is a $\GLtwo$ model 
$\krmod{M}=\langle W, \prec, \Vdash\rangle$
with $\krmod{M}\nVdash A$. 
By applying the usual Solovay construction for $\GL$, 
we find for all $a\in \krmod{M}$ a statement $L=a$ in the language of arithmetic, such that for all modal sentences $B$ not containing any occurrences of $\Box$:
\begin{equation}\label{loomad}
a\Vdash B\;\;\; \Rightarrow \;\;\; \PA\vdash L=a \to B^*,
\end{equation}
where $*$ is an arithmetical realization\footnote{$p^*$ is defined to be $\bigvee_{\krmod{M}, a\Vdash p}L=a$. For more details, see \cite{Sol76}, or the proof of Theorem \ref{GLTcompl} in the previous section.} mapping the modality $\triangle$
to square root provability, i.e.\ to $\Box_{\mathsf{r}}$. 
Now, using that $\krmod{M}\Vdash \triangle\triangle B\leftrightarrow \Box B$ on the modal side, and $\PA\vdash \Box_{\recfun{r}}\Box_{\recfun{r}} B^*\leftrightarrow \Box B^*$ on the arithmetical side, it is easy to check that (\ref{loomad}) can be extended to all $B\in \mathcal{L}_{\Box\triangle}$ (with $*$ mapping $\Box$ to ordinary provability). Using this, $\PA\nvdash A^{*}$ follows by the usual argument. 
\end{proof}

\section{Open Problems}
\begin{enumerate}
\item In the present paper we have shown that there exist two large groups of slow provability predicates of the form $\Box_{\thr{S}_n,\fast_{\varepsilon_0}}$: the predicates that are $\omega$-``roots'' and $\varepsilon_0$-``roots'' of the ordinary provability in $\PA$. For which ordinals $\alpha\in (\omega,\varepsilon_0)$ there exist  sequence $\thr{S}_n$ such that $\PA$ proves that and  $\bigcup\limits_{n\in \omega} \thr{S}_n=\PA$ and for every $\varphi$ we have $\Box_{\thr{S}_n,\fast_{\varepsilon_0}}^{\alpha}\varphi\mathrel{\leftrightarrow}\Box \varphi$?
\item In theorems \ref{Square_root_theorem} , \ref{omega_root_theorem}, and \ref{Epsilon_note_root_theorem} we have established  a correspondence between ordinary and slow provability. One could also considered slow variants of $\BoxI[n]$. Would the analogues of the mentioned theorem holds for $\BoxI[n]$ and its slow variant?
\item In Section \ref{ProvLogic_Section} we have proved that Lindstr\"om Logic is the provability logic of pairs of provability predicates from a rather large family. One could easily generalize Lindstr\"om logic to a polymodal case with linearly ordered family of modal connectives  by stating all the copies of axioms of Lindstr\"om logic, where $\triangle$ correspond to a modality with a smaller index and $\Box$ to a modality with larger index. Would the analogous of Theorem \ref{GLTcompl} holds for this logics?
\item  It is easy to see that any pair of provability predicates of the form $(\triangle,\triangle^{\alpha})$, where $\alpha\ge \omega$ satisfy the conditions of Theorem \ref{GLTcompl}. Is it true that for every pair of provability predicates $(\triangle,\Box)$ that satisfy conditions of Theorem \ref{GLTcompl} we can find elementary well-ordering and $\alpha$ from it such that $\Box$ is equivalent to $\triangle^{\alpha}$?
\item For which natural $n>2$ there exist a recursive $\recfun{f}$ such that $\PA$ proves that for every $\varphi$ we have $\Box_{\recfun{f}}^n\varphi\mathrel{\leftrightarrow}\Box\varphi$? 
\end{enumerate}


\subsubsection*{Acknowledgements}
We are grateful to Volodya Shavrukov whose insights were indispensable for proving Theorem \ref{GLTcompl}. We wish to thank Lev Beklemishev for inspiring discussions, and his advice on simplifying some of the arguments.

\bibliographystyle{model1-num-names}
\bibliography{bibliography}
\end{document}